\documentclass[11pt,a4paper,leqno]{article}
\usepackage{a4wide}
\usepackage{latexsym}
\usepackage{amsmath}
\usepackage{amssymb}
\newtheorem{defin}{Definition}
\newtheorem{lemma}{Lemma}
\newtheorem{prop}{Proposition}
\newtheorem{theo}{Theorem}

\newenvironment{proof}{\medskip\par\noindent{\bf Proof}}{\hfill $\Box$
\medskip\par}

\def\C{\mathbb{C}}
\def\N{\mathbb{N}}
\def\R{\mathbb{R}}

\begin{document}
\title{Multi-level Gevrey solutions of singularly perturbed linear partial differential equations}
\author{{\bf A. Lastra\footnote{The author is partially supported by the project MTM2012-31439 of Ministerio de Ciencia e
Innovacion, Spain}, S. Malek\footnote{The author is partially supported by the french ANR-10-JCJC 0105 project and the PHC
Polonium 2013 project No. 28217SG.}}\\
University of Alcal\'{a}, Departamento de F\'{i}sica y Matem\'{a}ticas,\\
Ap. de Correos 20, E-28871 Alcal\'{a} de Henares (Madrid), Spain,\\
University of Lille 1, Laboratoire Paul Painlev\'e,\\
59655 Villeneuve d'Ascq cedex, France,\\
{\tt alberto.lastra@uah.es}\\
{\tt Stephane.Malek@math.univ-lille1.fr }}
\maketitle
\thispagestyle{empty}
{ \small \begin{center}
{\bf Abstract}
\end{center}

We study the asymptotic behavior of the solutions related to a family of singularly perturbed linear partial differential equations in the complex domain. The analytic solutions obtained by means of a Borel-Laplace summation procedure are represented by a formal power series in the perturbation parameter. Indeed, the geometry of the problem gives rise to a decomposition of the formal and analytic solutions so that a multi-level Gevrey order phenomenon appears. This result leans on a Malgrange-Sibuya theorem in several Gevrey levels.

\medskip


}
\bigskip 

\section{Introduction}

We study a family of singularly perturbed linear partial differential equations of the following form
\begin{equation}\label{e1}
(\epsilon^{r_2}(t^{k+1}\partial_t)^{s_2}+a_2)(\epsilon^{r_1}(t^{k+1}\partial_t)^{s_1}+a_1)\partial_{z}^{S}X(t,z,\epsilon)
=\sum_{(s,\kappa_0,\kappa_1)\in\mathcal{S}}b_{\kappa_0\kappa_1}(z,\epsilon)t^s(\partial_{t}^{\kappa_0}\partial_{z}^{\kappa_1}X)(t,z,\epsilon),
\end{equation}
for given initial conditions
\begin{equation}\label{e2}
(\partial_{z}^{j}X)(t,0,\epsilon)=\phi_{i,j}(t,\epsilon),\quad 0\le j\le S-1,
\end{equation}
where $r_1$ and $r_2$ stand for nonnegative integers (i. e. they belong to $\mathbb{N}=\{0,1,...\}$), and $s_1,s_2$ are positive integers. We also fix $a_1,a_2\in\mathbb{C}^{\star}$. $\mathcal{S}$ consists of a finite subset of elements $(s,\kappa_0,\kappa_1)\in\mathbb{N}^{3}$. We assume that $S>\kappa_1$ for every $(s,\kappa_0,\kappa_1)\in\mathcal{S}$, and also that $b_{s,\kappa_0,\kappa_1}(z,\epsilon)$ belongs to the space of holomorphic functions in a neighborhood of the origin in $\mathbb{C}^2$, $\mathcal{O}\{z,\epsilon\}$. 

The initial data consist of holomorphic functions defined in a product of finite sectors with vertex at the origin.

The framework of our study is the asymptotic study of singularly perturbed Cauchy problems of the form
\begin{equation}\label{e66}
L(t,z,\partial_t,\partial_z,\epsilon)[u(t,z,\epsilon)]=0,
\end{equation}
where $L$ is a linear differential operator, for some given initial conditions $(\partial_z^{j}u)(t,0, \epsilon)=h_j(t,\epsilon)$, $0\le j\le \nu-1$ belonging certain functional spaces. Here, $\epsilon$ plays the role of a perturbation parameter near the origin and it turns out to be the variable in which asymptotic solutions are being obtained. There is a wide literature dealing with the case where $\epsilon$ is real, $L=\epsilon^mL_1(t,z,\partial_t,\partial_z)$ is acting on $\mathcal{C}^{\infty}(\mathbb{R}^{d})$ functions or Sovolev spaces $H^{s}(\mathbb{R}^d)$. For a survey on this topic, we refer to~\cite{ka}. 

On the other hand, the case for complex perturbation parameter $\epsilon$ has also been studied when solving partial differential equations; in particular, when dealing with solutions belonging to spaces of analytic functions for singularly perturbed partial differential equations which exhibit several singularities of different nature. On this direction, one can cite the work by M. Canalis-Durand, J. Mozo-Fern\'andez and R. Sch\"{a}fke~\cite{cms}, S. Kamimoto~\cite{kami}, the second author~\cite{ma1,ma2}, and the first and the second author and J. Sanz~\cite{lamasa1}. In this last work, the appearance of both, irregular and fuchsian singularities in the problem causes that the Gevrey type concerning the asymptotic representation of the formal solution varies with respect to a problem in which only one type of such singularities appears. 

The asymptotic behavior of the solution in the problem under study (\ref{e1}), (\ref{e2}) differs from the previous ones for the singularities are of different nature. Indeed, the appearance of two irregular singularities $t^{k+1}\partial_t$ perturbed by a certain power of $\epsilon$ enriches the accuracy of the information provided in the sense that different Gevrey orders can be distinguished.


The main aim in this work is to construct actual holomorphic solutions $X(t,z,\epsilon)$ of (\ref{e1}), (\ref{e2}) which are represented by the formal solution 
\begin{equation}\label{e78}
\hat{X}(t,z,\epsilon)=\sum_{\beta\ge0}H_{\beta}(t,z)\frac{\epsilon^{\beta}}{\beta!},
\end{equation}
where $H_{\beta}$ belongs to an adecquate space of functions. The solution is holomorphic in a domain of the form $\mathcal{T}\times \mathcal{U}\times\mathcal{E}$, where $\mathcal{T}$ and $\mathcal{E}$ are sectors of finite radius and vertex at the origin, and $\mathcal{U}$ is a neighborhood of the origin. In the asymptotic representation several Gevrey orders will appear.

The strategy followed is to study, for every fixed $\epsilon\in\mathcal{E}$, a singular Cauchy problem (see (\ref{e385}), (\ref{e389})) where $Y(t,z,\epsilon):=X(\epsilon^{-r}t,z,\epsilon)$ turns out to be its solution. Of course, the domain of definition of such a solution depends on the choice of $\epsilon\in\mathcal{E}$. More precisely, for every $\epsilon\in\mathcal{E}$ one finds a function
$$(T,z)\mapsto Y(T,z,\epsilon)=\sum_{\beta\ge0}Y_{\beta}(T,\epsilon)\frac{z^{\beta}}{\beta!}$$
defined in a sector of radius depending on $\epsilon$ and wide enough opening in the variable $T$ times a neighborhood of the origin (see Theorem~\ref{teo1}). Indeed, the function $T\mapsto Y_{\beta}(T,\epsilon)$ is constructed as the $m_{k}-$Laplace transform of $\tau\mapsto W_{\beta}(\tau,\epsilon)$ belonging to a well chosen Banach space (see Definition~\ref{defi1}). 

At this point, we have handled a slightly modified version of the classical Laplace transform which better fits our needs, and has already been used in other works in the framework of singularly perturbed Cauchy problems with vanishing initial data, such as~\cite{lama1}. 

It is worth noticing that some assumptions on the elements appearing on the equation of the singular Cauchy problem are made (see Assumption (D)) in order to be able to write the operators involved of some form. This idea is reproduced from~\cite{taya}. 

The coefficients of the formal power series 
$$W(\tau,z,\epsilon)=\sum_{\beta\ge0}W_{\beta}(\tau,\epsilon)\frac{z^{\beta}}{\beta!}$$ 
belong to some appropriate Banach space which depend on $\epsilon\in\mathcal{E}$; and $W(\tau,z,\epsilon)$ is constructed as the formal solution to the auxiliary Cauchy problem (\ref{e237}), (\ref{e238}) (see Proposition~\ref{prop2}). 

The solution $X(t,z,\epsilon)$ is written in the form
\begin{equation}\label{e94}
X(t,z,\epsilon)=\sum_{\beta\ge0}\int_{L_{\gamma}}W_{\beta}(u,\epsilon)e^{-\left(\frac{u}{\epsilon^rt}\right)^{k}}\frac{du}{u}\frac{z^{\beta}}{\beta!},
\end{equation}
where $L_{\gamma}=[0,\infty)e^{i\gamma}$ for some $\gamma\in[0,2\pi)$.

Regarding the singularities appearing, one realizes that the geometry of the problem is crucial when approaching the auxiliary and the initial problems. Indeed, the singularities in equation (\ref{e237}) come from the zeroes in the variable $\tau$ of the equations $(k\tau^k)^{s_2}+a_2=0$ and $\epsilon^{r_1-s_1rk}(k\tau^k)^{s_1}+a_1=0$.
The first equation provides fixed singularities which do not depend on $\epsilon$ whilst the second equation provides singularities that converge to the origin with $\epsilon$. The geometry associated to this phenomenon is described in Section~\ref{seccion1} and also in Assumption (B) in more detail. As a matter of fact, for every $\beta\ge0$, $\tau\mapsto W_{\beta}(\tau,\epsilon)$ is a holomorphic function defined in a neighborhood of the origin which can be extended along an infinite sector (common for every $\epsilon\in\mathcal{E}$). However, this initial neighborhood of 0 varies with $\epsilon$; all its complex numbers within a certain range of directions and modulus larger than a function of $\epsilon$ which tends to 0 with $|\epsilon|\to 0$ are being removed from it.

Regarding the asymptotic representation of the analytic solution we study problem (\ref{e1}), (\ref{e2}) with the perturbation parameter lying in different sectors $\mathcal{E}_{i}$, $i=1,...,\nu-1$, where $(\mathcal{E}_{i})_{1\le i\le \nu-1}$ provides a good covering at 0 (see Definition~\ref{defi3}). By means of a Ramis-Sibuya type theorem with two levels we were able to estimate the difference of two consecutive solutions by deforming the integration path of the $m_k-$Laplace transform in (\ref{e94}). This deformation is made accordingly with the geometry explained above so that, if some particular argument lies in between the integration path of two consecutive solutions, then the Gevrey order within the asymptotic representation is altered.

We should mention that a similar phenomenon of parametric multilevel Gevrey asymptotics has been observed recently by K. Suzuki and Y. Takei in \cite{suta} and Y. Takei in \cite{ta} for WKB solutions of the Schr\"{o}dinger equation
$$ \epsilon^{2}\psi''(z) = (z - \epsilon^{2}z^{2})\psi(z) $$
which possess 0 as fixed turning point and $z_{\epsilon}=\epsilon^{-2}$ as movable turning point. We stress the fact a resembling Ramis-Sibuya type theorem is used in this work.

As a consequence, there exists a common $\hat{X}$ for every $i=1,...,\nu-1$ of the form (\ref{e78}) which can be splitted in the form 
$$\hat{X}(t,z,\epsilon)=a(t,z,\epsilon)+\hat{X}^{1}(t,z,\epsilon)+\hat{X}^{2}(t,z,\epsilon),$$
where $a$ is a convergent series on some neighborhood of the origin, such that the solution $X_{i}(t,z,\epsilon)$ is given by
$$X_i(t,z,\epsilon)=a(t,z,\epsilon)+X_i^{1}(t,z,\epsilon)+X_i^{2}(t,z,\epsilon),$$
where $X_i^{j}$ admits $\hat{X}^{j}$ as its Gevrey asymptotic expansion in $\mathcal{E}_{i}$ of order $\hat{r}_{j}$ for $j=1,2$ (see Theorem~\ref{teopral}).

The layout of the work reads as follows.

In Section~\ref{seccion1}, we describe a parameter depending Banach space of holomorphic functions and describe some geometry associated to the domain of definition of the elements in such space. We also describe the behavior of the elements in it under certain operators. In Section~\ref{seccion3}, we study the formal solution of the auxiliary Cauchy problem (\ref{e237}), (\ref{e238}) with coefficients being elements in the Banach space described in the previous section. After recalling some definitions and properties on the $k-$Borel-Laplace summability procedure in Section~\ref{sec41}, we provide the solutions of a singular Cauchy problem (\ref{e385}), (\ref{e389}) which conform the support of the solution for the main problem in our work (\ref{e500}), (\ref{e504}). Finally, we estimate the difference of two solutions of the main problem in the intersection of their domain of definition in the perturbation parameter (see Theorem~\ref{teo2}) and obtain, by means of a Ramis-Sibuya theorem with two levels (see Section~\ref{secrs}), a formal solution and a decomposition of both the analytic and the formal solution of the problem in two terms so that each term in the formal solution represents the corresponding term in the analytic one under certain Gevrey type asymptotics (see Theorem~\ref{teopral}).

\section{Banach spaces functions with exponential decay}\label{seccion1}

Let $\rho_0>0$. We denote $D(0,\rho_0)$ the open disc in $\C$, centered at 0 and with radius $\rho_0$. For $d\in\R$, we consider an unbounded sector $\{z\in\C:|\arg(z)-d|<\delta_{1}\}$, for some $\delta_1>0$, which is denoted by $S_d$.

Let $\mathcal{E}$ be an open and bounded sector with vertex at the origin. We put 
$$\mathcal{E}=\left\{\epsilon\in\C: |\epsilon|<r_{\mathcal{E}},\theta_{1,\mathcal{E}}<\arg(\epsilon)<\theta_{2,\mathcal{E}}\right\},$$ 
for some $r_{\mathcal{E}}>0$ and $0\le\theta_{1,\mathcal{E}}<\theta_{2,\mathcal{E}}<2\pi$. 

Let $\delta_2>0$. For every $\epsilon\in\mathcal{E}$, we consider the open domain $\Omega(\epsilon):=(S_{d}\cup D(0,\rho_0))\setminus \Omega_{1}(\epsilon)$, where $\Omega_{1}(\epsilon)$ turns out to be a finite collection of sets of the form $\{\tau\in\C: |\tau|>\rho(|\epsilon|),|\arg(\tau)-d_{\mathcal{E}}|<\delta_{2}\}$, where $0\le d_{\mathcal{E}}<2\pi$ is a real number depending on $\mathcal{E}$, and $x\in(0,r_{\mathcal{E}})\mapsto \rho(x)$ is a monotone increasing function with $\rho(x)\to0$ when $x\to0$. We give more technical details on the construction of this set afterwards, in Assumption (B.2), not to interrupt the reasonings. We only remark now that $S_d$ and $\Omega_1(\epsilon)$ are such that $S_d\cap \Omega_1(\epsilon)=\emptyset$ for every $\epsilon\in\mathcal{E}$.

Throughout this work, $b$ and $\sigma$ are fixed positive real numbers with $b>1$, whilst $k\ge2$ stands for a fixed integer. 

\begin{defin}\label{defi1}
Let $\epsilon\in\mathcal{E}$ and $r\in\mathbb{Q}$, $r>0$.
 
For every $\beta\ge0$, we consider the vector space $F_{\beta,\epsilon,\Omega(\epsilon)}$ of holomorphic functions $\tau\mapsto h(\tau,\epsilon)$ defined in $\Omega(\epsilon)$ such that
$$\left\|h(\tau,\epsilon) \right\|_{\beta,\epsilon,\Omega(\epsilon)}:=\sup_{\tau\in\Omega(\epsilon)}\left\{\frac{1+\left|\frac{\tau}{\epsilon^{r}}\right|^{2k}}{\left|\frac{\tau}{\epsilon^{r}}\right|}\exp\left(-\sigma r_{b}(\beta)\left|\frac{\tau}{\epsilon^{r}}\right|^{k}\right)\left|h(\tau,\epsilon)\right|\right\}<\infty,$$
where $r_{b}(\beta)=\sum_{n=0}^{\beta}\frac{1}{(n+1)^b}$. One can check that the pair $(F_{\beta,\epsilon,\Omega(\epsilon)},\left\|\cdot\right\|_{\beta,\epsilon,\Omega(\epsilon)})$ is a Banach space.
\end{defin}

\textbf{Assumption (A):}
Let $a_2\in\C$ with $a_2\neq0$, and let $s_2$ be a positive integer. We assume:
\begin{enumerate}
\item[$(A.1)$]
$$\arg(\tau)\neq\frac{\pi (2j+1)+\arg(a_2)}{ks_2},\quad j=0,...,ks_2-1,$$
for every $\tau\in\overline{S_{d}}\setminus\{0\}$.
\item[$(A.2)$] $\rho_0<\frac{|a_{2}|^{1/(ks_2)}}{2k^{1/k}}$.
\end{enumerate}

The aim of the previous assumption is to avoid the roots of the function $\tau\mapsto (k\tau^{k})^{s_2}+a_2$ when $\tau$ lies among the elements in $\Omega(\epsilon)$ for every $\epsilon\in\mathcal{E}$. This statement is clarified in the following

\begin{lemma}\label{lema1}
Under Assumption (A), there exists a constant $C_1>0$ (which only depends on $k$, $s_2$, $a_2$) such that
$$\left|\frac{1}{(k\tau^{k})^{s_2}+a_2}\right|\le C_1,$$
for every $\epsilon\in\mathcal{E}$ and every $\tau\in\Omega(\epsilon)$.
\end{lemma}
\begin{proof}
This proof follows analogous steps as the one of Lemma 1 in~\cite{lamasa1}. Let $\epsilon\in\mathcal{E}$. 

On the one hand, it is direct to check from Assumption (A.2) that any root of $\tau\mapsto (k\tau^{k})^{s_2}+a_2$ keeps positive distance to $D(0,\rho_0)$. This entails this distance provides an upper bound when substituting $D(0,\rho_0)$ by $D(0,\rho_0)\setminus\Omega_{1}(\epsilon)$ for every $\epsilon\in\mathcal{E}$. 

On the other hand, one has that
$$\frac{1}{(k\tau^{k})^{s_2}+a_2}=\sum_{j=0}^{ks_2-1}\frac{A_{j}}{\tau-\frac{a_2^{1/(ks_2)}e^{i\pi\left(\frac{2j+1}{ks_2}\right)}}{k^{1/k}}},$$
where
$$A_{j}=\frac{1}{a_2ks_2}e^{-i\pi\left(\frac{ks_2-1}{ks_2}\right)(2j+1)}\frac{a_2^{1/(ks_2)}}{k^{1/k}},$$
for every $j=0,...,ks_2-1$. Taking into account Assumption (A.1), there exists a constant $C_{11}>0$, which does not depend on $\epsilon\in\mathcal{E}$, satisfying
$$\left|\tau-\frac{a_{2}^{1/(ks_2)}e^{i\pi\left(\frac{2j+1}{ks_2}\right)}}{k^{1/k}}\right|\ge C_{11},$$
for every $\tau \in S_{d}$ and all $j=0,...,ks_2-1$. 

Both statements yield to the conclusion. 
\end{proof}

We now give more detail on the construction of the set $\Omega(\epsilon)$ for each $\epsilon\in\mathcal{E}$.

\textbf{Assumption (B):} Let $a_1\in\C$ with $a_1\neq0$ . Let $r_1$ be a nonnegative integer, and $r_2,s_1$ positive integers. We assume:
\begin{enumerate} 
\item[$(B.1)$] $s_1r_2-s_2r_1>0$.

\item[$(B.2)$]For every $\epsilon\in\mathcal{E}$, the set $\Omega_{1}(\epsilon)$ is constructed as follows:
$$\Omega_{1}(\epsilon):=\bigcup_{j=0}^{ks_1-1} \left\{\tau\in\C: |\tau|\ge \rho(|\epsilon|),|\arg(\tau)-d_{\mathcal{E},j}|<\delta_2 \right\},$$
where 
$$\rho(x)=\frac{|a_1|^{1/(k s_1)}x^{\frac{s_1r_2-s_2r_1}{s_1s_2k}}}{2k^{1/k}},\qquad x\ge0,$$
\begin{equation}\label{e116}
d_{\mathcal{E},j}=\frac{1}{ks_1}\left(\pi(2j+1)+\arg(a_1)+\frac{s_1r_2-s_2r_1}{s_2}\left(\frac{\theta_{1,\mathcal{E}}+\theta_{2,\mathcal{E}}}{2}\right)\right),
\end{equation}
for every $j=0,...,ks_1-1$,
and
\begin{equation}\label{e117}
\delta_{2}>\frac{s_1r_2-s_2r_1}{ks_1s_2}\left(\theta_{2,\mathcal{E}}-\theta_{1,\mathcal{E}}\right).
\end{equation}
\end{enumerate}

Assumption (B) is concerned with the nature of the roots of the function 
\begin{equation}\label{e121}
\tau\mapsto \epsilon^{r_1-s_1rk}(k\tau^{k})^{s_1}+a_1,
\end{equation}
with 
\begin{equation}\label{e131}
r:=\frac{r_2}{s_2k}.
\end{equation}
The dynamics of the singularities involved in the equation to study is related to the first item in the previous assumption. More precisely, these tend to 0 with the perturbation parameter $\epsilon$. The second enunciate in Assumption (B) is concerned with the distance of $\Omega_1(\epsilon)$ to the roots of (\ref{e121}). Indeed, one can choose a positive lower bound for this distance which does not depend on $\epsilon\in\mathcal{E}$.



\begin{lemma}\label{lema2}
Let $\epsilon\in\mathcal{E}$. Under Assumption (B), there exists a constant $C_2>0$ (which only depends on $k, s_1,s_2,r_1,r_2,a_1$ and which is independent of $\epsilon\in\mathcal{E}$) such that
$$\left|\frac{1}{\epsilon^{r_1-s_1rk}(k\tau^{k})^{s_1}+a_1}\right|\le C_2,$$
for every $\tau\in\Omega(\epsilon)$.
\end{lemma}
\begin{proof}
The proof of this result follows analogous steps as the corresponding one of Lemma~\ref{lema1}. Let $\epsilon\in\mathcal{E}$. One can write
$$\frac{1}{\epsilon^{r_1-s_1rk}(k\tau^{k})^{s_1}+a_1}=\sum_{j=0}^{ks_{1}-1}\frac{B_{j}(\epsilon)}{\tau-\frac{e^{i\pi\left(\frac{2j+1}{ks_1}\right)}a_1^{1/(ks_1)}}{k^{1/k}\epsilon^{\frac{r_1-s_1rk}{ks_1}}}},$$
where 
$$B_{j}(\epsilon)=\frac{1}{a_1ks_1}e^{-i\pi\left(\frac{ks_1-1}{ks_1}\right)(2j+1)}\frac{a_1^{1/(k s_1)}}{k^{1/k}}\epsilon^{\frac{s_1r_2-s_2r_1}{ks_1s_2}},$$
for every $j=0,...,ks_1-1$. Indeed, for all $j=0,...,ks_1-1$ one has
$$\frac{B_{j}(\epsilon)}{\tau-\frac{e^{i\pi\left(\frac{2j+1}{ks_1}\right)}a_1^{1/(ks_1)}}{k^{1/k}\epsilon^{\frac{r_1-s_1rk}{ks_1}}}}=\frac{a_1^{1/(ks_1)}e^{-i\pi\left(\frac{ks_1-1}{ks_1}\right)(2j+1)}}{k^{1/k}ks_1a_1(\tau\epsilon^{-\frac{s_1r_2-s_2r_1}{ks_1s_2}}-k^{-1/k}e^{i\pi\left(\frac{2j+1}{ks_1}\right)}a_1^{1/(ks_1)})}.$$
At this point, it is sufficient to prove that the distance from $\tau\epsilon^{-\frac{s_1r_2-s_2r_1}{ks_1s_2}}$ to $k^{-1/k}e^{i\pi\left(\frac{2j+1}{ks_1}\right)}a_1^{1/(ks_1)}$ is upper bounded by a constant for every $\tau\in\Omega(\epsilon)$, which does not depend on $\epsilon$. Let $\tau(\epsilon)\in\C$ be satisfying
\begin{equation}\label{e163}
\tau(\epsilon)\epsilon^{-\frac{s_1r_2-s_2r_1}{ks_1s_2}}-k^{-1/k}e^{i\pi\left(\frac{2j+1}{ks_1}\right)}a_1^{1/(ks_1)}=0.
\end{equation}
Regarding the construction of $\Omega_1(\epsilon)$, the distance from $\tau(\epsilon)$ to $\Omega_1(\epsilon)$ might be attained at the complex points in $\overline{\Omega_1(\epsilon)}$ with arguments given by $d_{\mathcal{E},j}\pm\delta_2$ or at the points in $\overline{\Omega_{1}(\epsilon)}$ of modulus equal to $\rho(|\epsilon|)$. In the first case, this distance is positive and does not depend on $\epsilon$ as it can be deduced from (\ref{e116}) and (\ref{e117}). In the second case, the minimum distance is attained at $\tau(\epsilon)/2$. Taking into account (\ref{e163}) one derives that
$$ \left|\frac{\tau(\epsilon)}{2}\epsilon^{-\frac{s_1r_2-s_2r_1}{ks_1s_2}}-k^{-1/k}e^{i\pi\left(\frac{2j+1}{ks_1}\right)}a_1^{1/(ks_1)}\right|=\left|\frac{-a_1^{1/(ks_1)}e^{i\pi\left(\frac{2j+1}{ks_1}\right)}}{2k^{1/k}}\right|=\frac{|a_1|^{1/(k s_1)}}{2k^{1/k}}>0,
$$
which does not depend on $\epsilon$. The conclusion is achieved from this point.
\end{proof}

Assumption (B.1) is substituted by the incoming Assumption (B.1)'. It deals with the existence of attainable directions $d\in\R$ in such a way that $S_d\cap \left(\cup_{\epsilon\in\mathcal{E}}\Omega_1(\epsilon)\right)=\emptyset$. Indeed, for this purpose one aims that
$$\arg(\tau)\neq\frac{1}{ks_1}\left[\pi(2j+1)+\arg(a_1)+\frac{s_1r_2-s_2r_1}{s_2}\arg(\epsilon)\right],$$
for any $j=0,...,ks_1-1$, $\epsilon\in\mathcal{E}$ and all $\tau\in \overline{S_d}\setminus\{0\}$.

This entails that
\begin{equation}\label{e152}
ks_1\arg(\tau)\notin\left(\pi(2j+1)+\arg(a_1)+\frac{s_1r_2-s_2r_1}{s_2}\theta_{1,\mathcal{E}},\pi(2j+1)+\arg(a_1)+\frac{s_1r_2-s_2r_1}{s_2}\theta_{2,\mathcal{E}}\right),
\end{equation}
for any $j=0,...,ks_1-1$. 
The overlapping of two consecutive sectors in $\Omega_{1}(\epsilon)$ for some $\epsilon\in\mathcal{E}$ would imply such $d$ could not exist. Regarding (\ref{e152}), the existence of possible choices for direction $d$ implies undertaking the following

\textbf{Assumption (C):}
$$\theta_{2,\mathcal{E}}-\theta_{1,\mathcal{E}}<\frac{2\pi s_2}{s_1r_2-s_2r_1}.$$
which implies

\textbf{Assumption (B.1)':} $$s_1r_2-s_2r_1>s_2$$  

>From now on, we substitute Assumption (B.1) by Assumption (B.1)', which is more restrictive.

The next lemmas are devoted to the behavior of the elements in the latter Banach space introduced in Definition~\ref{defi1} under some operators, and its continuity. 

\begin{lemma}
Let $\epsilon\in\mathcal{E}$ and $\beta$ be a nonnegative integer. For every bounded continuous function $g(\tau)$ on $\Omega(\epsilon)$ such that $M_g:=\sup_{\tau\in\Omega(\epsilon)}|g(\tau)|$ does not depend on $\epsilon\in\mathcal{E}$, then 
$$\left\|g(\tau)h(\tau,\epsilon)\right\|_{\beta,\epsilon,\Omega(\epsilon)}\le M_{g}\left\|h(\tau,\epsilon)\right\|_{\beta,\epsilon,\Omega(\epsilon)},$$
for every $h\in F_{\beta,\epsilon,\Omega(\epsilon)}$.
\end{lemma}
\begin{proof}
It s a direct consecuence of the definition of the space  $F_{\beta,\epsilon,\Omega(\epsilon)}$.
\end{proof}

\begin{prop}\label{prop1}
Let $\epsilon\in\mathcal{E}$ and $r\in\mathbb{Q}$, $r>0$. We consider real numbers $\nu\ge0$ and $\xi\ge-1$. Let $S\ge1$ be a positive integer, $r$ a positive rational number, and let $\alpha<\beta$ be nonnegative integers. Then, there exists a constant $C_3>0$ (depending on $\alpha$, $S$, $\beta$, $\xi$, $\nu$  and which does not depend on $\epsilon$) with
$$\left\|\tau^{k}\int_{0}^{\tau^{k}}(\tau^{k}-s)^{\nu}s^{\xi}f(s^{1/k},\epsilon)ds \right\|_{\beta,\epsilon,\Omega(\epsilon)}\le C_3|\epsilon|^{rk(2+\nu+\xi)}\left(\frac{(\beta+1)^b}{\beta-\alpha}\right)^{\nu+\xi+3}\left\|f(\tau,\epsilon)\right\|_{\alpha,\epsilon,\Omega(\epsilon)},$$ 
for every $f\in F_{\alpha,\epsilon,\Omega(\epsilon)}$.
\end{prop}
\begin{proof}
Let $f\in F_{\alpha,\epsilon,\Omega(\epsilon)}$. For every $\tau\in\Omega(\epsilon)$, the segment $[0,\tau^{k}]$ is contained in $\Omega(\epsilon)$ for it is a star domain with respect to 0. By definition, we have
$$\left\|\tau^{k}\int_{0}^{\tau^{k}}(\tau^{k}-s)^{\nu}s^{\xi}f(s^{1/k},\epsilon)ds \right\|_{\beta,\epsilon,\Omega(\epsilon)}$$
$$=\sup_{\tau\in\Omega(\epsilon)}\left\{\frac{1+\left|\frac{\tau}{\epsilon^{r}}\right|^{2k}}{\left|\frac{\tau}{\epsilon^{r}}\right|}\exp\left(-\sigma r_{b}(\beta)\left|\frac{\tau}{\epsilon^{r}}\right|^{k}\right)|\tau|^{k}\left|\int_{0}^{\tau^{k}}(\tau^{k}-s)^{\nu}s^{\xi}f(s^{1/k},\epsilon)ds\right|\right\}$$
$$\le  \sup_{\tau\in\Omega(\epsilon)}\left\{\frac{1+\left|\frac{\tau}{\epsilon^{r}}\right|^{2k}}{\left|\frac{\tau}{\epsilon^{r}}\right|}e^{-\sigma r_{b}(\beta)\left|\frac{\tau}{\epsilon^{r}}\right|^{k}}|\tau|^{k}\int_{0}^{|\tau|^{k}}\frac{1+\frac{s^2}{|\epsilon^{r}|^{2k}}}{\frac{s^{1/k}}{|\epsilon|^{r}}}e^{-\sigma r_{b}(\alpha)\frac{s}{|\epsilon^{r}|^{k}}}|f(s^{1/k}e^{\sqrt{-1}k\arg(\tau)},\epsilon)|\right.$$
$$\left. (|\tau|^{k}-s)^{\nu}s^{\xi}\frac{\frac{s^{1/k}}{|\epsilon|^{r}}}{1+\frac{s^2}{|\epsilon^{r}|^{2k}}}\exp\left(\sigma r_{b}(\alpha)\frac{s}{|\epsilon^{r}|^{k}}\right)ds\right\}.$$
Taking into account that for every $s\in[0,|\tau|^{k}]$ one has
$$\exp\left(-\sigma r_{b}(\beta)\left|\frac{\tau}{\epsilon^{r}}\right|^{k}\right)\exp\left(\sigma r_{b}(\alpha)\frac{s}{|\epsilon^{r}|^{k}}\right)\le \exp\left(-\sigma(r_b(\beta)-r_b(\alpha))\left|\frac{\tau}{\epsilon^r}\right|^{k}\right)=:e(\left|\frac{\tau}{\epsilon^r}\right|^{k}),$$
and by the change of variable $s=|\epsilon^{r}|^{k}h$, the last expression can be upper bounded by
$$\left\|f(\tau,\epsilon)\right\|_{\alpha,\epsilon,\Omega(\epsilon)}\sup_{\tau\in\Omega(\epsilon)}\left\{\frac{1+\left|\frac{\tau}{\epsilon^{r}}\right|^{2k}}{\left|\frac{\tau}{\epsilon^{r}}\right|}e(\left|\frac{\tau}{\epsilon^r}\right|^{k})|\tau|^{k}\int_{0}^{\frac{|\tau|^{k}}{|\epsilon^{r}|^{k}}} (|\tau|^{k}-|\epsilon^r|^{k}h)^{\nu}|\epsilon^{r}|^{k\xi}h^{\xi}\frac{h^{1/k}}{1+h^2}|\epsilon^{r}|^{k}dh\right\}$$
$$\le |\epsilon|^{rk(2+\nu+\xi)}\left\|f(\tau,\epsilon)\right\|_{\alpha,\epsilon,\Omega(\epsilon)}\sup_{x\ge0}B(x),$$
where 
$$B(x)=\frac{1+x^2}{x^{1/k}}e(x)x\int_{0}^{x}\frac{h^{1/k}}{1+h^2}(x-h)^{\nu}h^{\xi}dh.$$
It only rests to provide a constant upper bound for $B(x)$ in order to conclude the proof. One can estimate
$$B(x)\le (1+x^2)e(x)x^{\nu+1}\int_{0}^{x}\frac{h^{\xi}}{1+h^2}dh=B_2(x).$$

>From standard calculations one arrives at 
$$B_2(x)\le C_{31}x^{\nu+\xi+3}\exp\left(-\sigma(r_{b}(\beta)-r_{b}(\alpha))x\right)$$
for some $C_{13}>0$. The standard estimates
$$x^{m_{1}}e^{-m_{2}x}\le \left(\frac{m_1}{m_2}\right)^{m_{1}}e^{-m_1},\quad x\ge0$$
for $m_1,m_2>0$ and the definition of $r_{b}$, one concludes that
$$B_{2}(x)\le C_{32}(\nu,\xi,\sigma)\left(\frac{(\beta+1)^b}{\beta-\alpha}\right)^{\nu+\xi+3}.$$
The result follows directly from here. 
\end{proof}

\section{An auxiliary Cauchy problem}\label{seccion3}

In this section we study the existence of a formal solution for the forthcoming auxiliary Cauchy problem (\ref{e237}), (\ref{e238}). After assuring the existence of a formal solution to this problem as a formal power series in $z$, we provide estimates on its coefficients in terms of the norms in Definition~\ref{defi1}. 

We keep the notations of Section~\ref{seccion1}, the construction of $\Omega(\epsilon)$ for every $\epsilon\in\mathcal{E}$ and also the values of the constants $r_1,r_2,s_1,s_2,r,k,b,\sigma,a_1$ and $a_2$ hold.

Let $S$ be a positive integer and $\mathcal{S}$ be a finite subset of $\N^{3}$. For every $(s,\kappa_0,\kappa_1)\in\mathcal{S}$, $b_{\kappa_0\kappa_1}(z,\epsilon)$ is a holomorphic and bounded function in a product of discs centered at the origin. We put
\begin{equation}\label{e243}
b_{\kappa_0\kappa_1}(z,\epsilon)=\sum_{\beta\ge0}b_{\kappa_0\kappa_1\beta}(\epsilon)\frac{z^{\beta}}{\beta!},
\end{equation}
for some holomorphic and bounded functions $b_{\kappa_0\kappa_1\beta}(\epsilon)$ defined on some neighborhood of the origin, which is common for every $\beta\ge0$. We assume that $b_{\kappa_0\kappa_10}(\epsilon)\equiv 0$ for every $(\kappa_0,\kappa_1,s)\in\mathcal{S}$.

We now make the following assumption on the elements of $\mathcal{S}$.

\textbf{Assumption (D):}
For every $(s,\kappa_0,\kappa_1)\in\mathcal{S}$ we have that $S>\kappa_0$, $S>\kappa_1$, $\kappa_0\ge1$. Moreover, there exists a nonnegative integer $\delta_{\kappa_0}\ge k$ such that $$s=\kappa_0(k+1)+\delta_{\kappa_0},$$
and that $S>\left\lfloor b\left(\frac{\delta_{\kappa_0}}{k}+\kappa_0\right)\right\rfloor +1$.

We also consider  $A_{\kappa_0,p}\in\C$ for every $(s,\kappa_0,\kappa_1)\in\mathcal{S}$ and $1\le p\le \kappa_0$.

It is worth mentioning that $\epsilon\in\mathcal{E}$ remains fixed through the whole section, so that the solution of the auxiliary Cauchy problem depends on $\epsilon$.

For every fixed $\epsilon\in\mathcal{E}$ we consider the following Cauchy problem
\begin{equation}\label{e237}
((k\tau^{k})^{s_2}+a_2)(\epsilon^{r_1-s_1rk}(k\tau^k)^{s_1}+a_1)\partial_{z}^{S}W(\tau,z,\epsilon)
\end{equation}
$$=\sum_{(s,\kappa_0,\kappa_1)\in\mathcal{S}}b_{\kappa_0\kappa_1}(z,\epsilon)\epsilon^{-r(s-\kappa_0)}\left[\frac{\tau^{k}}{\Gamma\left(\frac{\delta_{\kappa_0}}{k}\right)}\int_{0}^{\tau^{k}}(\tau^{k}-s)^{\frac{\delta_{\kappa_0}}{k}-1}(ks)^{\kappa_0}\partial_{z}^{\kappa_1}W(s^{1/k},z,\epsilon)\frac{ds}{s} \right.$$
$$\left.+\sum_{1\le p\le \kappa_0-1}A_{\kappa_0,p}\frac{\tau^{k}}{\Gamma\left(\frac{\delta_{\kappa_0+k(\kappa_0-p)}}{k}\right)}\int_{0}^{\tau^{k}}(\tau^k-s)^{\frac{\delta_{\kappa_0}+k(\kappa_0-p)}{k}-1}(ks)^{p}\partial_{z}^{\kappa_1}W(s^{1/k},z,\epsilon)\frac{ds}{s}\right],$$   
for given initial data
\begin{equation}\label{e238}
(\partial_{z}^{j}W)(\tau,0,\epsilon)=W_{j}(\tau,\epsilon)\in F_{j,\epsilon,\Omega(\epsilon)},\quad 0\le j\le S-1.
\end{equation}

\begin{prop}\label{prop2}
Under Assumptions (A), (B), (C) on the geometric configuration of our framework, and under Assumption (D), there exists a formal power series solution of (\ref{e237}),(\ref{e238}),
\begin{equation}\label{e258}
W(\tau,z,\epsilon)=\sum_{\beta\ge0}W_{\beta}(\tau,\epsilon)\frac{z^{\beta}}{\beta!}\in F_{\beta,\epsilon,\Omega(\epsilon)}[[z]],
\end{equation}
such that $W_\beta(\tau,\epsilon)\in F_{\beta,\epsilon,\Omega(\epsilon)}$ for every $\beta\ge0$. Moreover, these coefficients satisfy the recursion formula
\begin{equation}\label{e257}
\frac{W_{\beta+S}(\tau,\epsilon)}{\beta!}=\frac{1}{((k\tau^k)^{s_2}+a_2)(\epsilon^{r_1-s_1rk}(k\tau^k)^{s_1}+a_1)}\sum_{(s,\kappa_0,\kappa_1)\in\mathcal{S}}\sum_{\alpha_0+\alpha_1=\beta}\frac{b_{\kappa_0\kappa_1\alpha_0}(\epsilon)}{\alpha_0!}\epsilon^{-r(s-\kappa_0)}\times
\end{equation}
$$\times\left[\frac{\tau^{k}}{\Gamma\left(\frac{\delta_{\kappa_0}}{k}\right)}\int_{0}^{\tau^{k}}(\tau^k-s)^{\frac{\delta_{\kappa_0}}{k}-1}(ks)^{\kappa_0}\frac{W_{\alpha_1+\kappa_1}(s^{1/k},\epsilon)}{\alpha_1!}\frac{ds}{s}\right.$$
$$\left.+\sum_{1\le p \le \kappa_0-1}A_{\kappa_0,p}\frac{\tau^{k}}{\Gamma\left(\frac{\delta_{\kappa_0}+k(\kappa_0-p)}{k}\right)}\int_{0}^{\tau^{k}}(\tau^k-s)^{\frac{\delta_{\kappa_0}+k(\kappa_0-p)}{k}-1}(ks)^{p}\frac{W_{\alpha_1+\kappa_1}(s^{1/k},\epsilon)}{\alpha_1!}\frac{ds}{s} \right],$$
for every $\beta\ge0$, $\tau\in\Omega(\epsilon)$.
\end{prop}
\begin{proof}
Let $\beta\ge0$, $\epsilon\in\mathcal{E}$ and $\tau\in\Omega(\epsilon)$. The recursion formula in (\ref{e257}) is directly obtained after substitution of (\ref{e258}) in the equation (\ref{e237}). It is worth remarking that from the construction of $\Omega(\epsilon)$ leading to Lemma~\ref{lema1} and Lemma~\ref{lema2}, the function $W_{\beta+S}(\tau,\epsilon)$ is well defined and holomorphic in $\Omega(\epsilon)$ for every $\beta\ge0$. We now prove that $W_{\beta}(\tau,\epsilon)\in F_{\beta,\epsilon,\Omega(\epsilon)}$ for every $\beta\ge0$. 

This is valid for $0\le \beta\le S-1$ due to (\ref{e238}) holds.

Let $$w_{\beta}(\epsilon):=\left\|W_{\beta}(\tau,\epsilon)\right\|_{\beta,\epsilon,\Omega(\epsilon)}.$$
Taking $\left\|\cdot\right\|_{\beta+S,\epsilon,\Omega(\epsilon)}$ on both sides of the recursion formula (\ref{e257}), one obtains that
$$\frac{w_{\beta+S}(\epsilon)}{\beta!}\le\frac{1}{|(k\tau^k)^{s_2}+a_2||\epsilon^{r_1-s_1rk}(k\tau^k)^{s_1}+a_1|}\sum_{(s,\kappa_0,\kappa_1)\in\mathcal{S}}\sum_{\alpha_0+\alpha_1=\beta}\frac{|b_{\kappa_0\kappa_1\alpha_0}(\epsilon)|}{\alpha_0!}|\epsilon|^{-r(s-\kappa_0)}\times
$$
$$\times\left[\left\|\frac{\tau^{k}}{\Gamma\left(\frac{\delta_{\kappa_0}}{k}\right)}\int_{0}^{\tau^{k}}(\tau^k-s)^{\frac{\delta_{\kappa_0}}{k}-1}(ks)^{\kappa_0}\frac{W_{\alpha_1+\kappa_1}(s^{1/k},\epsilon)}{\alpha_1!}\frac{ds}{s}\right\|_{\beta+S,\epsilon,\Omega(\epsilon)}\right.$$
$$\left.+\sum_{1\le p \le \kappa_0-1}|A_{\kappa_0,p}|\left\|\frac{\tau^{k}}{\Gamma\left(\frac{\delta_{\kappa_0}+k(\kappa_0-p)}{k}\right)}\int_{0}^{\tau^{k}}(\tau^k-s)^{\frac{\delta_{\kappa_0}+k(\kappa_0-p)}{k}-1}(ks)^{p}\frac{W_{\alpha_1+\kappa_1}(s^{1/k},\epsilon)}{\alpha_1!}\frac{ds}{s}\right\|_{\beta+S,\epsilon,\Omega(\epsilon)} \right].
$$
>From Lemma~\ref{lema1}, Lemma~\ref{lema2}, and Proposition~\ref{prop1}, the right-hand side of the previous inequality can be upper bounded so that
\begin{equation}\label{e284}
\frac{w_{\beta+S}(\epsilon)}{\beta!}\le C_4\sum_{(s,\kappa_0,\kappa_1)\in\mathcal{S}}\sum_{\alpha_0+\alpha_1=\beta}\frac{|b_{\kappa_0\kappa_1\alpha_0}(\epsilon)|}{\alpha_0!}|\epsilon|^{-r(s-\kappa_0)}\times
\end{equation}
$$\times\left[\frac{k^{\kappa_0}|\epsilon|^{rk(\frac{\delta_{\kappa_0}}{k}+\kappa_0)}}{\Gamma\left(\frac{\delta_{\kappa_0}}{k}\right)}\left(\frac{(\beta+S+1)^b}{\beta+S-\alpha_1-\kappa_1}\right)^{\frac{\delta_{\kappa_0}}{k}+\kappa_0+1}\frac{w_{\alpha_{1}+\kappa_{1}}(\epsilon)}{\alpha_1!}\right.$$
$$\left. +\sum_{1\le p\le \kappa_0-1}|A_{\kappa_0,p}|\frac{k^{p}|\epsilon|^{rk(\frac{\delta_{\kappa_0}}{k}+\kappa_0)}}{\Gamma\left(\frac{\delta_{\kappa_0}+k(\kappa_0-p)}{k}\right)}\left(\frac{(\beta+S+1)^b}{\beta+S-\alpha_1-\kappa_1}\right)^{\frac{\delta_{\kappa_0}+k(\kappa_0-p)}{k}+p+1}\frac{w_{\alpha_{1}+\kappa_1}(\epsilon)}{\alpha_1!}\right],$$
for some $C_4>0$. Observe that $\left\|g(\tau,\epsilon)\right\|_{\alpha,\epsilon,\Omega(\epsilon)}\ge \left\|g(\tau,\epsilon)\right\|_{\gamma,\epsilon,\Omega(\epsilon)}$ whenever $\alpha\le \gamma$. From Assumption (D), one has
$$|\epsilon|^{-r(s-\kappa_0)+rk(\frac{\delta_{\kappa_0}}{k}+\kappa_0)}=1.$$

Let $M_{\kappa_0\kappa_1\beta}>0$ be such that $|b_{\kappa_0\kappa_1\beta}(\epsilon)|\le M_{\kappa_0\kappa_1\beta}$ for all $\beta\ge0$ and for every $(s,\kappa_0,\kappa_1)\in\mathcal{S}$. We define $B_{\kappa_0\kappa_1}(z)=\sum_{\beta\ge0}M_{\kappa_0\kappa_1\beta}\frac{z^{\beta}}{\beta!}$. From the assumptions made on $b_{\kappa_0\kappa_1}$ there exist $D_1,D_2>0$ such that $M_{\kappa_0\kappa_1\beta}\le D_1 D_2^{\beta}\beta!$ for every $\beta\ge0$. The function $B_{\kappa_0\kappa_1}(z)$ turns out to be a holomorphic and bounded function on some neighborhood of the origin.

The terms $\left(\frac{(\beta+S+1)^b}{\beta+S-\alpha_1-\kappa_1}\right)^{\frac{\delta_{\kappa_0}+k(\kappa_0-p)}{k}+p+1}$ appearing in (\ref{e284}) can be upper bounded by
$$C_{41}\beta(\beta-1)\cdots(\beta-\left\lfloor b(\frac{\delta_{\kappa_0}}{k}+\kappa_0)\right\rfloor)(\beta-\left\lfloor b(\frac{\delta_{\kappa_0}}{k}+\kappa_0)\right\rfloor +1)$$ for some $C_{41}>0$.

We consider the Cauchy problem
$$\partial_{x}^{S}u(x,\epsilon)=C_4C_{41}\sum_{(s,\kappa_0,\kappa_1)\in\mathcal{S}}B_{\kappa_0\kappa_1}(x)\left[\frac{k^{\kappa_0}}{\Gamma\left(\frac{\delta_{\kappa_0}}{k}\right)}+\sum_{1\le p\le \kappa_0-1}|A_{\kappa_0,p}|\frac{k^{p}}{\Gamma\left(\frac{\delta_{\kappa_0}+k(\kappa_0-p)}{k}\right)}\right]$$
\begin{equation}\label{e285}
\partial_{x}^{\kappa_1}x^{\left\lfloor b(\frac{\delta_{\kappa_0}}{k}+\kappa_0)\right\rfloor +1}\partial_{x}^{\left\lfloor b(\frac{\delta_{\kappa_0}}{k}+\kappa_0)\right\rfloor +1}u(x,\epsilon),
\end{equation}
with initial conditions
\begin{equation}\label{e286}
(\partial_{x}^{j}u)(0,\epsilon)=w_{j}(\epsilon),\quad 0\le j \le S-1.
\end{equation}
The problem (\ref{e285}),(\ref{e286}) has a unique formal solution 
$$u(x,\epsilon)=\sum_{\beta\ge0}u_{\beta}(\epsilon)\frac{x^{\beta}}{\beta!}\in\R[[x]].$$
Moreover, its coefficients satisfy the recursion formula 
\begin{equation}\label{e302}
\frac{u_{\beta+S}(\epsilon)}{\beta!}= C_4C_{41}\sum_{(s,\kappa_0,\kappa_1)\in\mathcal{S}}\sum_{\alpha_0+\alpha_1=\beta}\frac{M_{\kappa_0\kappa_1\alpha_0}}{\alpha_0!}|\epsilon|^{-r(s-\kappa_0)}\frac{\beta!}{\left(\beta- \left\lfloor b(\frac{\delta_{\kappa_0}}{k}+\kappa_0)\right\rfloor\right)!}
\end{equation}
$$\times\left[\frac{k^{\kappa_0}|\epsilon|^{rk(\frac{\delta_{\kappa_0}}{k}+\kappa_0)}}{\Gamma\left(\frac{\delta_{\kappa_0}}{k}\right)}\frac{u_{\alpha_{1}+\kappa_{1}}(\epsilon)}{\alpha_1!}+\sum_{1\le p\le \kappa_0-1}|A_{\kappa_0,p}|\frac{k^{p}|\epsilon|^{rk(\frac{\delta_{\kappa_0}}{k}+\kappa_0)}}{\Gamma\left(\frac{\delta_{\kappa_0}+k(\kappa_0-p)}{k}\right)}\frac{u_{\alpha_{1}+\kappa_1}(\epsilon)}{\alpha_1!}\right].$$

>From the initial conditions of the problem (\ref{e285}), (\ref{e286}), one gets that $u_{j}(\epsilon)=w_{j}(\epsilon)$ for $0\le j\le S-1$. Regarding (\ref{e284}) and (\ref{e302}) one has
$$w_{\beta}(\epsilon)\le u_{\beta}(\epsilon),$$
for every $\beta\ge0$.

>From the classical theory of existence of solutions of ODEs, there exists $\rho_1>0$ such that whenever $w_{j}(\epsilon)<\rho_1$ for every $0\le j\le S-1$, one has that the unique formal solution of (\ref{e285}), (\ref{e286}), $u(x,\epsilon)=\sum_{\beta\ge0}u_{\beta}(\epsilon)\frac{x^{\beta}}{\beta!}$ belongs to $\C\{x\}$, with a radius of convergence $Z_0>0$. Regarding the previous steps one can affirm that this radius of convergence does not depend on the choice of $\epsilon\in\mathcal{E}$.

This yields the existence of $M>0$ such that
$$\sum_{\beta\ge0}u_{\beta}(\epsilon)\frac{Z_0^{\beta}}{\beta!}<M,$$
for every $\epsilon\in\mathcal{E}$ which entails $0<u_{\beta}(\epsilon)<MZ_0^{\beta}\beta!$ for every $\beta\ge0$.
The result is attained for
\begin{equation}\label{e319}
\left\|W_{\beta}(\tau,\epsilon)\right\|_{\beta,\epsilon,\Omega(\epsilon)}=w_{\beta}(\epsilon)\le u_{\beta}(\epsilon)\le MZ_0^{\beta}\beta!<\infty,
\end{equation}
for every $\beta\ge0$.
\end{proof}


\section{ Analytic solutions of a singular Cauchy problem}

\subsection{Laplace transform and asymptotic expansions}\label{sec41}

In the present section we give some details on the $k$-Borel summability procedure of formal power series with coefficients belonging to a complex Banach space. This is a slightly modified version of the more classical one, which can be found in detail in~\cite{ba2}, Section 3.2. This novel version entails a different behavior of Borel and Laplace transforms with respect to the operators involved, which has already been used in the previous work \cite{lama1} procuring fruitful results in the framework of Cauchy problems depending upon a complex perturbation parameter, with vanishing initial data. We refer to~\cite{lama1} for further details. 

\begin{defin}\label{defi2}
Let $k\ge1$ be an integer. Let $(m_{k}(n))_{n\ge1}$ be the sequence 
$$m_{k}(n)=\Gamma\left(\frac{n}{k}\right)=\int_{0}^{\infty}t^{\frac{n}{k}-1}e^{-t}dt,\qquad n\ge1.$$
Let $(\mathbb{E},\left\|\cdot\right\|_{\mathbb{E}})$ be a complex Banach space. We say a formal power series 
$$\hat{X}(T)=\sum_{n=1}^{\infty}a_{n}T^{n}\in T\mathbb{E}[[T]]$$ is $m_{k}$-summable with respect to $T$ in the direction $d\in[0,2\pi)$ if the following assertions hold:
\begin{enumerate}
\item There exists $\rho>0$ such that the $m_{k}$-Borel transform of $\hat{X}$, $\mathcal{B}_{m_{k}}(\hat{X})$, is absolutely convergent for $|\tau|<\rho$, where
$$\mathcal{B}_{m_{k}}(\hat{X})(\tau)=\sum_{n=1}^{\infty}\frac{a_{n}}{\Gamma\left(\frac{n}{k}\right)}\tau^{n}\in\tau\mathbb{E}[[\tau]].$$
\item The series $\mathcal{B}_{m_{k}}(\hat{X})$ can be analytically continued in a sector $S=\{\tau\in\C^{\star}:|d-\arg(\tau)|<\delta\}$ for some $\delta>0$. In addition to this, the extension is of exponential growth of order $k$ in $S$, meaning that there exist $C,K>0$ such that
$$\left\|\mathcal{B}_{m_{k}}(\hat{X})(\tau)\right\|_{\mathbb{E}}\le Ce^{K|\tau|^{k}},\quad \tau\in S.$$ 
\end{enumerate}
Under these assumptions, the vector valued Laplace transform of $\mathcal{B}_{m_{k}}(\hat{X})$ along direction $d$ is defined by
$$\mathcal{L}_{m_{k}}^{d}\left(\mathcal{B}_{m_{k}}(\hat{X})\right)(T)=k\int_{L_{\gamma}}\mathcal{B}_{m_{k}}(\hat{X})(u)e^{-(u/T)^k}\frac{du}{u},$$
where $L_{\gamma}$ is the path parametrized by $u\in[0,\infty)\mapsto ue^{i\gamma}$,for some appropriate direction $\gamma$ depending on $T$, such that $L_{\gamma}\subseteq S$ and $\cos(k(\gamma-\arg(T)))\ge\Delta>0$ for some $\Delta>0$.

The function $\mathcal{L}_{m_{k}}^{d}(\mathcal{B}_{m_{k}}(\hat{X})$ is well defined and turns out to be a holomorphic and bounded function in any sector of the form $S_{d,\theta,R^{1/k}}=\{T\in\C^{\star}:|T|<R^{1/k},|d-\arg(T)|<\theta/2\}$, for some $\frac{\pi}{k}<\theta<\frac{\pi}{k}+2\delta$ and $0<R<\Delta/K$. This function is known as the $m_k$-sum of the formal power series $\hat{X}(T)$ in the direction $d$.
\end{defin}

The main aim in the present work is to study the asymptotic behavior of the solutions of equation (\ref{e1}), (\ref{e2}) and relate them to its formal solution by means of Gevrey asymptotic expansions. The following are some elementary properties concerning the $m_k$-sums of formal power series which will be crucial in our procedure.

1) The function $\mathcal{L}_{m_{k}}^{d}(\mathcal{B}_{m_{k}}(\hat{X}))(T)$ admits $\hat{X}(T)$ as its Gevrey asymptotic expansion of order $1/k$ with respect to $t$ in $S_{d,\theta,R^{1/k}}$. More precisely, for every $\frac{\pi}{k}<\theta_1<\theta$, there exist $C,M>0$ such that 
$$\left\|\mathcal{L}^{d}_{m_{k}}(\mathcal{B}_{m_{k}}(\hat{X}))(T)-\sum_{p=1}^{n-1}a_{p}T^{p}\right\|_{\mathbb{E}}\le CM^{n}\Gamma(1+\frac{n}{k})|T|^{n},$$
for every $n\ge2$ and $T\in S_{d,\theta,R^{1/k}}$. Watson's lemma (see Proposition 11 p.75 in \cite{ba}) allows us to affirm that $\mathcal{L}^{d}_{m_{k}}(\mathcal{B}_{m_{k}}(\hat{X})(T)$ is unique provided that the opening $\theta_1$ is larger than $\frac{\pi}{k}$. 

2) The set of holomorphic functions having Gevrey asymptotic expansion of order $1/k$ on a sector with values in $\mathbb{E}$ turns out to be a differential algebra (see Theorem 18, 19 and 20 in~\cite{ba}). This, and the uniqueness provided by Watson's lemma provide some properties on $m_{k}$-summable formal power series in direction $d$. 

We now assume $\mathbb{E}$ to be a Banach algebra for the product $\star$. Let $\hat{X}_{1}$, $\hat{X}_{2}\in T\mathbb{E}[[T]]$ be $m_{k}$-summable formal power series in direction $d$. Let $q_1\ge q_2\ge1$ be integers. We assume that $ \hat{X}_{1}+\hat{X}_{2}$, $\hat{X}_{1}\star \hat{X}_{2}$ and $T^{q_1}\partial_{T}^{q_2}\hat{X}_{1}$, which are elements of $T\mathbb{E}[[T]]$, are $m_{k}$-summable in direction $d$. Then, one has
$$\mathcal{L}_{m_{k}}^{d}(\mathcal{B}_{m_{k}}(\hat{X}_{1}))(T)+\mathcal{L}_{m_{k}}^{d}(\mathcal{B}_{m_{k}}(\hat{X}_{2}))(T)=\mathcal{L}_{m_{k}}^{d}(\mathcal{B}_{m_{k}}(\hat{X}_{1}+\hat{X}_{2}))(T),$$
$$\mathcal{L}_{m_{k}}^{d}(\mathcal{B}_{m_{k}}(\hat{X}_{1}))(T)\star \mathcal{L}_{m_{k}}^{d}(\mathcal{B}_{m_{k}}(\hat{X}_{2}))(T)=\mathcal{L}_{m_{k}}^{d}(\mathcal{B}_{m_{k}}(\hat{X}_{1}\star\hat{X}_{2}))(T),$$
$$T^{q_1}\partial_{T}^{q_2}\mathcal{L}^{d}_{m_{k}}(\mathcal{B}_{m_{k}}(\hat{X}_{1}))(T)=\mathcal{L}_{m_{k}}^{d}(\mathcal{B}_{m_{k}}(T^{q_1}\partial_{T}^{q_2}\hat{X}_{1}))(T),$$
for every $T\in S_{d,\theta,R^{1/k}}$.

The next proposition is written without proof for it can be found in \cite{lama1}, Proposition 6.

\begin{prop}\label{prop3}
Let $\hat{f}(t)=\sum_{n\ge1}f_nt^n\in\mathbb{E}[[t]]$, where $(\mathbb{E},\left\|\cdot\right\|_{\mathbb{E}})$ is a Banach algebra. Let $k,m\ge1$ be integers. The following formal identities hold.
$$\mathcal{B}_{m_{k}}(t^{k+1}\partial_{t}\hat{f}(t))(\tau)=k\tau^{k}\mathcal{B}_{m_{k}}(\hat{f}(t))(\tau),$$
$$\mathcal{B}_{m_{k}}(t^{m}\hat{f}(t))(\tau)=\frac{\tau^{k}}{\Gamma\left(\frac{m}{k}\right)}\int_{0}^{\tau^{k}}(\tau^{k}-s)^{\frac{m}{k}-1}\mathcal{B}_{m_{k}}(\hat{f}(t))(s^{1/k})\frac{ds}{s}.$$
\end{prop}

\subsection{Analytic solutions of a singular Cauchy problem}

Let $S\ge1$ be an integer. We also consider a nonnegative integer $r_1$ and positive integers $r_2,s_1,s_2,k$. The positive real number $r$ is defined by (\ref{e131}). Let $a_1,a_2\in\C^{\star}$ and assume $\mathcal{E}$, $S_{d}$ (and with it $\delta_1$) and $D(0,\rho_0)$ are constructed in the shape of Section~\ref{seccion1}, for some $d\in[0,2\pi)$, and some $\rho_0>0$ so that Assumptions (A), (B) and (C) hold. We also fix $\gamma\in[0,2\pi)$ such that $\R_{+}e^{i\gamma}\subseteq S_{d}\cup\{0\}$.

Let $\mathcal{S}$ be as in Section~\ref{seccion3}, which satisfies Assumption (D). For every $(s,\kappa_0,\kappa_1)\in\mathcal{S}$ we consider an holomorphic and bounded function $b_{\kappa_{0}\kappa_1}(z,\epsilon)$ defined in a product of discs with center at the origin which can be written as in (\ref{e243}), and $A_{\kappa_0,p}\in\C$ for every $1\le p\le \kappa_0-1$.

We point out that the perturbation parameter remains fixed in this singular Cauchy problem, as in the auxiliary Cauchy problem in Section~\ref{seccion3}.

For every $\epsilon\in\mathcal{E}$ we consider the following Cauchy problem
\begin{equation}\label{e385}
((T^{k+1}\partial_{T})^{s_2}+a_2)(\epsilon^{r_1-s_1rk}(T^{k+1}\partial_{T})^{s_1}+a_1)\partial_{z}^{S}Y(T,z,\epsilon)
\end{equation}
$$=\sum_{(s,\kappa_0,\kappa_1)\in\mathcal{S}}b_{\kappa_0\kappa_1}(z,\epsilon)\epsilon^{-r(s-\kappa_0)}T^{s}(\partial_{T}^{\kappa_0}\partial_{z}^{\kappa_1}Y)(T,z,\epsilon),$$
for given initial conditions
\begin{equation}\label{e389}
(\partial^{j}_{z}Y)(T,0,\epsilon)=Y_j(T,\epsilon),\quad 0\le j\le S-1.
\end{equation}
The initial conditions $(Y_{j}(T,\epsilon))_{0\le j\le S-1}$ are constructed as follows: for every $0\le j\le S-1$, let $\tau\mapsto W_{j}(\tau,\epsilon)$ be a holomorphic function defined in $\Omega(\epsilon)$. Moreover, assume there exists $M_0>0$ such that
\begin{equation}\label{e399}
\sup_{\epsilon\in\mathcal{E}}\left\|W_{j}(\tau,\epsilon)\right\|_{j,\epsilon,\Omega(\epsilon)}<M_0,\quad 0\le j\le S-1.
\end{equation}
Then, we define
\begin{equation}\label{e400}
Y_{j}(T,\epsilon):=\mathcal{L}_{m_{k}}^{d}(W_{j}(\tau,\epsilon))(T),
\end{equation}
where the Laplace transform is taken with respect to the variable $\tau$, along the direction $d$. Observe from Definition~\ref{defi1} and Definition~\ref{defi2} that for every fixed $\epsilon\in\mathcal{E}$, the definition in (\ref{e400}) makes sense, providing a function $T\mapsto Y_{j}(T,\epsilon)$ which is well defined and holomorphic for all $T=|T|e^{i\theta}$ such that $\cos(k(\gamma-\theta))\ge\Delta$, for some $\Delta>0$, and $|T|\le|\epsilon|^{r}\frac{\Delta^{1/k}}{(\sigma \xi(b))^{1/k}}$, where $\xi(b)=\sum_{n\ge0}\frac{1}{(n+1)^b}$.

In the incoming result, we provide the solution of (\ref{e385}), (\ref{e389}) by means of the properties of Laplace transform and the solution of the auxiliary Cauchy problem studied in Section~\ref{seccion3}.

\begin{theo}\label{teo1}
Let $\epsilon\in\mathcal{E}$. Under the assumptions made at the beginning of the present section the problem (\ref{e385}), (\ref{e389}) admits a holomorphic solution $(T,z)\mapsto Y(T,z,\epsilon)$ defined in $$S_{d,\theta,|\epsilon|^{r}\left(\frac{\Delta}{\sigma\xi(b)}\right)^{1/k}}\times D(0,1/Z_0),$$ for some $Z_0>0$ and some $\theta>\pi/k$, where
\begin{equation}\label{e411}
S_{d,\theta,|\epsilon|^{r}\left(\frac{\Delta}{\sigma\xi(b)}\right)^{1/k}}=\left\{T\in\C^{\star}:|T|\le |\epsilon|^{r}\left(\frac{\Delta}{\sigma\xi(b)}\right)^{1/k},|\arg(T)-d|<\frac{\theta}{2}\right\}.
\end{equation}
\end{theo}
\begin{proof}
Taking into account Assumption (D), one can write $T^{s}\partial_{T}^{\kappa_0}$ in the form $T^{\delta_{\kappa_0}}T^{\kappa_0(k+1)}\partial_{T}^{\kappa_0},$
for every $(s,\kappa_0,\kappa_1)\in\mathcal{S}$, for some nonnegative integers $\delta_{\kappa_0}$. By means of the formula appearing in page 40 of~\cite{taya}, one can expand the previous operators in the form
\begin{equation}\label{e415}
T^{\delta_{\kappa_0}}T^{\kappa_0(k+1)}\partial_{T}^{\kappa_0}=T^{\delta_{\kappa_0}}\left((T^{k+1}\partial_T)^{\kappa_0}+\sum_{1\le p\le \kappa_0-1}A_{\kappa_0,p}T^{k(\kappa_0-p)}(T^{k+1}\partial_{T})^p\right),
\end{equation}
for some complex numbers $A_{\kappa_0,p}\in\C$. Regarding (\ref{e415}), equation (\ref{e385}) is transformed into 
\begin{equation}\label{e419}
((T^{k+1}\partial_{T})^{s_2}+a_2)(\epsilon^{r_1-s_1rk}(T^{k+1}\partial_{T})^{s_1}+a_1)\partial_{z}^{S}Y(T,z,\epsilon)
\end{equation}
$$ 
=\sum_{(s,\kappa_0,\kappa_1)\in\mathcal{S}}b_{\kappa_0\kappa_1}(z,\epsilon)\epsilon^{-r(s-\kappa_0)}T^{\delta_{\kappa_0}}\left[(T^{k+1}\partial_{T})^{\kappa_0}+\sum_{1\le p\le \kappa_0-1}A_{\kappa_0,p}T^{k(\kappa_0-p)}(T^{k+1}\partial_{T})^p\right]\partial_{z}^{\kappa_1}Y(T,z,\epsilon).
$$

One can apply the formal Borel transform $\mathcal{B}_{m_{k}}$ with respect to the variable $T$ at both sides of equation (\ref{e415}). The properties of this formal operator shown in Proposition~\ref{prop3} turn equation (\ref{e415}) into (\ref{e237}), with $W(\tau,z,\epsilon)=\mathcal{B}_{m_{k}}(Y(T,z,\epsilon))(\tau)$.

Regarding (\ref{e399}), one has $W_{j}\in F_{j,\epsilon,\Omega(\epsilon)}$ for $0\le j\le S-1$. One can apply Proposition~\ref{prop2} to the Cauchy problem with equation (\ref{e237}) and initial data given by
\begin{equation}\label{e432}
(\partial_{z}^{j}W)(\tau,0,\epsilon)=W_{j}(\tau,\epsilon),\quad 0\le j\le S-1
\end{equation}
to arrive at the existence of a formal solution of this problem of the form
\begin{equation}\label{e433}
\sum_{\beta\ge0}W_{\beta}(\tau,\epsilon)\frac{z^\beta}{\beta!}\in F_{\beta,\epsilon,\Omega(\epsilon)}[[z]].
\end{equation}
Moreover, regarding (\ref{e319}) there exist $Z_0,M>0$ such that
\begin{equation}\label{e434}
|W_{\beta}(\tau,\epsilon)|\le MZ_0^{\beta}\beta!\frac{\left|\frac{\tau}{\epsilon^r}\right|}{1+\left|\frac{\tau}{\epsilon^r}\right|^{2k}}\exp\left(\sigma r_b(\beta)\left|\frac{\tau}{\epsilon^r}\right|^{k}\right), \quad \beta\ge0,
\end{equation}
for every $\tau\in\Omega(\epsilon)$.

If we write $T=|T|e^{i\theta}$, we deduce that
\begin{align*}
\left|k\int_{L_{\gamma}}W_{\beta}(u,\epsilon)e^{\left(\frac{u}{T}\right)^{k}}\frac{du}{u}\right|&\le k\int_{0}^{\infty}|W_{\beta}(se^{i\gamma},\epsilon)|e^{-\frac{s^{k}}{|T|^{k}}\cos(k(\gamma-\arg(T)))}ds\\
&\le kMZ_0^{\beta}\beta! \int_{0}^{\infty}\exp(\left[\frac{\sigma\xi(b)}{|\epsilon|^{rk}}-\frac{\Delta}{|T|^k}\right]s^k)ds,
\end{align*}
for every $\beta\ge0$. 

This entails the function $\mathcal{L}_{m_{k}}^{d}(W_{\beta}(\tau,\epsilon))(T)$ is well defined for $T\in S_{d,\theta,|\epsilon|^{r}\left(\frac{\Delta}{\sigma\xi(b)}\right)^{1/k}},$ for every $\frac{\pi}{k}<\theta<\frac{\pi}{k}+2\delta$.

Moreover, 
$$(T,z)\mapsto Y(T,z,\epsilon):=\sum_{\beta\ge0}\mathcal{L}_{m_{k}}^{d}(W_{\beta}(\tau,\epsilon))(T)\frac{z^{\beta}}{\beta!}$$
defines a holomorphic function on $S_{d,\theta,|\epsilon|^{r}\left(\frac{\Delta}{\sigma\xi(b)}\right)^{1/k}}\times D(0,\frac{1}{Z_0})$, and it turns out to be a solution of the problem (\ref{e385}), (\ref{e389}) from the properties of Laplace transform in 2), Section~\ref{sec41} and the fact that (\ref{e433}) is a formal solution of (\ref{e237}), (\ref{e432}).

\end{proof}

\section{Formal series solutions and multi-level Gevrey asymptotic expansions in a complex parameter for a Cauchy problem}

This section is devoted to the study of the formal and analytic solutions of the main problem in the present work. The analytic solution is approximated by the formal solution in the perturbation parameter near the origin following different Gevrey levels which depend on the nature and location of the singular points involved. One may find two different situations depending on the geometry of the problem: that in which only the singularities not depending on the perturbation parameter are involved, and other situation in which a moving singularity makes appearance. This last one depends on the perturbation parameter and makes the singularity tend to the origin when the parameter vanishes. 

Let $r_1$ be a nonnegative integer, and $r_2,s_1,s_2,k$ be positive integers. We also fix $a_1,a_2\in\C^{\star}$. We define $r$ as in (\ref{e131}).

We first recall the notion of a good covering and justify the geometric choices involved in the framework of our problem.

\begin{defin}
Let $(\mathcal{E}_{i})_{0\le i\le \nu-1 }$ be a finite family of open sectors such that $\mathcal{E}_{i}$ has its vertex at the origin and finite radius $r_{\mathcal{E}_{i}}>0$ for every $0\le i\le \nu-1$. We say this family conforms a good covering in $\C^{\star}$ if $\mathcal{E}_{i}\cap\mathcal{E}_{i+1}\neq\emptyset$ for $0\le i\le \nu-1$ (we put $\mathcal{E}_{\nu}:=\mathcal{E}_{0}$) and $\cup_{0\le i\le \nu-1}\mathcal{E}_{i}=\mathcal{U}\setminus\{0\}$ for some neighborhood of the origin $\mathcal{U}$. 
\end{defin}

Without loss of generality, one can consider $r_{\mathcal{E}_{i}}:=r_{\mathcal{E}}$ for every $0\le i\le \nu-1$, for some positive real number $r_{\mathcal{E}}$, for our study is local at 0.

\begin{defin}\label{defi3}
Let $(\mathcal{E}_{i})_{0\le i\le \nu-1}$ be a good covering in $\C^{\star}$.
For every $0\le i\le \nu-1$, we assume 
$$\mathcal{E}_{i}=\{\epsilon\in\C^{\star}:|\epsilon|<r_{\mathcal{E}},\theta_{1,\mathcal{E}_{i}}<\arg(\epsilon)<\theta_{2,\mathcal{E}_{i}}\},$$
for some $r_{\mathcal{E}}>0$ and $0\le\theta_{1,\mathcal{E}_{i}}<\theta_{2,\mathcal{E}_{i}}<2\pi$. We write $d_{\mathcal{E}_{i}}$ for the bisecting direction of $\mathcal{E}_{i}$, $(\theta_{1,\mathcal{E}_{i}}+\theta_{2,\mathcal{E}_{i}})/2$. Let $\mathcal{T}$ be an open sector with vertex at 0 and finite radius, say $r_{\mathcal{T}}>0$. We also fix a family of open sectors
$$S_{d_{i},\theta,r_{\mathcal{E}}^{r}r_{\mathcal{T}}}=\left\{t\in\mathbb{C}^{\star}: |t|\le r_{\mathcal{E}}^{r}r_{\mathcal{T}},|d_i-\arg(t)|<\frac{\theta}{2} \right\},$$
with $d_{i}\in[0,2\pi)$ for $0\le i\le \nu-1$, and $\pi/k<\theta<\pi/k+\delta$, for some small enough $\delta>0$, under the following properties:
\begin{enumerate}
\item one has $\arg\left(d_i\right)\neq \frac{\pi(2j+1)+\arg(a_2)}{ks_2}$, for every $j=0,...,ks_2-1$.
\item one has $|\arg(d_{i})-d_{\mathcal{E}_{i},j}|>\delta_{2i}$, for $j=0,...,ks_1-1$, where $\delta_{2i}:=\frac{s_1r_2-s_2r_1}{ks_1s_2}\left(\theta_{2,\mathcal{E}_{i}}-\theta_{1,\mathcal{E}_{i}}\right)$, and $d_{\mathcal{E}_{i},j}=\frac{1}{ks_1}\left(\pi(2j+1)+\arg(a_1)+\frac{s_1r_2-s_2r_1}{s_2}\left(\frac{\theta_{1,\mathcal{E}_{i}}+\theta_{2,\mathcal{E}_{i}}}{2}\right)\right)$.
\item for every $0\le i\le \nu-1$, $t\in\mathcal{T}$ and $\epsilon\in\mathcal{E}_{i}$, one has $\epsilon^{r}t\in S_{d_{i},\theta,r_{\mathcal{E}}^{r}r_{\mathcal{T}}}$.
\end{enumerate}
Under the previous settings, we say the family $\{(S_{d_{i},\theta,r_{\mathcal{E}}^{r}r_{\mathcal{T}}})_{0\le i\le \nu-1},\mathcal{T}\}$ is associated to the good covering $(\mathcal{E}_{i})_{0\le i\le \nu-1}$.
\end{defin}

\textbf{Remark:} The previous construction is feasible under suitable choices for the elements involved. For example, if $\mathcal{T}$ is bisected by the positive real line and has a small enough opening, one can choose the constants in the definition of $\delta_{2i}$ such that $\delta_{2i}$ allows the third condition in the previous definition to be satisfied for every $0\le i \le \nu-1$ without falling into a forbidden direction. From Assumption (C), these forbidden directions do not cover $[0,2\pi)$.

Let us consider a good covering in $\C^{\star}$, $(\mathcal{E}_{i})_{0\le i\le \nu-1}$. In the following, we identify the first element $\mathcal{E}_{0}$ with $\mathcal{E}_{\nu}$.

Let $S\ge1$ be an integer. We also consider a finite subset $\mathcal{S}$ of $\N^3$, and for every $(s,\kappa_0,\kappa_1)\in\mathcal{S}$, let $b_{\kappa_0\kappa_1}(z,\epsilon)$ be as stated in Section~\ref{seccion3}, under the form (\ref{e243}).

For each $0\le i\le \nu-1$, we study the Cauchy problem
\begin{equation}\label{e500}
(\epsilon^{r_2}(t^{k+1}\partial_t)^{s_2}+a_2)(\epsilon^{r_1}(t^{k+1}\partial_t)^{s_1}+a_1)\partial_{z}^{S}X_{i}(t,z,\epsilon)
\end{equation}
$$=\sum_{(s,\kappa_0,\kappa_1)\in\mathcal{S}}b_{\kappa_0\kappa_1}(z,\epsilon)t^s(\partial_{t}^{\kappa_0}\partial_{z}^{\kappa_1}X_{i})(t,z,\epsilon),$$
for given initial conditions
\begin{equation}\label{e504}
(\partial_{z}^{j}(X_{i}))(t,0,\epsilon)=\phi_{i,j}(t,\epsilon),\quad 0\le j\le S-1,
\end{equation}
where the functions $\phi_{i,j}(t,\epsilon)$ are constructed in the following way:

Let $\{(S_{d_{i},\theta,r^{r}_{\mathcal{E}r_{\mathcal{T}}}}),\mathcal{T}\}$ be a family associated to the good covering $(\mathcal{E}_{i})_{0\le i\le \nu-1}$. For the sake of simplicity in the notation, we will denote $S_{d_{i},\theta,r_{\mathcal{E}^{r}r_{\mathcal{T}}}}$ by $S_{d_{i}}$ from now on, for every $0\le i\le\nu-1$. 

Let $j\in\{0,...,S-1\}$ and $i\in\{0,...,\nu-1\}$. We consider the construction in Section~\ref{seccion1} for the sets $\Omega(\epsilon)$, for a common sector $S_{d_{i}}$ for every $\epsilon\in\mathcal{E}_i$ and define $W_{ij}(\tau,\epsilon)$ such that: 

\begin{itemize}
\item[a)] For every $\epsilon\in\mathcal{E}_{i}$, the function $\tau\mapsto W_{i,j}(\tau,\epsilon)$ is an element in $F_{j,\epsilon,\Omega(\epsilon)}$, with 
\begin{equation}\label{nosesiesnecesaria}
\left\|W_{i,j}(\tau,\epsilon)\right\|_{j,\epsilon,\Omega(\epsilon)}<M_0, 
\end{equation}
for some $M_0>0$.
\item[b)] The function $(\tau, \epsilon)\mapsto W_{i,j}(\tau,\epsilon)$ is a holomorphic function in $\cup_{\epsilon\in\mathcal{E}_{i}}\Omega(\epsilon)\times \mathcal{E}_{i}$.
\item[c)] The function $W_{i,j}(\tau,\epsilon)$ coincides with $W_{i+1,j}(\tau,\epsilon)$ in the domain $\cup_{\epsilon\in(\mathcal{E}_{i}\cap\mathcal{E}_{i+1})}\Omega(\epsilon)\times(\mathcal{E}_{i}\cap\mathcal{E}_{i+1})$. 
\end{itemize}

Let $\gamma_{i}\in[0,2\pi)$ be chosen in such a way that the set $L_{\gamma_{i}}:=\R_{+}e^{\gamma_{i}\sqrt{-1}}\subseteq S_{d}\cup\{0\}$. Then, we define 
\begin{equation}\label{e525}
\phi_{i,j}(t,\epsilon)=Y_{i,j}(\epsilon^r t,\epsilon):=k\int_{L_{\gamma_{i}}}W_{i,j}(u,\epsilon)e^{-\left(\frac{u}{\epsilon^r t}\right)^{k}}\frac{du}{u},
\end{equation}
for every $(t,\epsilon)\in\mathcal{T}\times\mathcal{E}_{i}$. Regarding a), $\phi_{i,j}$ is well defined and from b) one has $\phi_{i,j}(t\epsilon,\epsilon)$ turns out to be a holomorphic function in $\mathcal{T}\times\mathcal{E}_{i}$.

The next assumption is more restrictive than Assumption (B.1)'. We adopt it and substitute (B.1)' for it in Assumption (B), for reasons that will be explained in the proof of Theorem~\ref{teo2}.


We are in conditions to construct the analytic solutions for the problem (\ref{e500}), (\ref{e504}).

\begin{theo}\label{teo2}
Let the initial data (\ref{e504}) be constructed as above. Under Assumptions (A), (B) and (C) on the geometry of the problem, and under Assumption (D) on the constants involved, the problem (\ref{e500}), (\ref{e504}) has a holomorphic and bounded solution $X_{i}(t,z,\epsilon)$ on $(\mathcal{T}\cup D(0,h'))\times D(0,R_0)\times \mathcal{E}_{i}$, for every $0\le i\le \nu-1$, for some $R_0,h'>0$. Moreover, there exist $0<h''<h'$, $K,M>0$ (not depending on $\epsilon$), such that
\begin{equation}\label{e536}
\sup_{\stackrel{t\in\mathcal{T}\cap D(0,h'')}{z\in D(0,\rho_0/2)}}\left|X_{i+1}(t,z,\epsilon)-X_{i}(t,z,\epsilon)\right|\le K\exp\left(-\frac{M}{|\epsilon|^{\hat{r}_i}}\right),
\end{equation}
for every $\epsilon\in\mathcal{E}_{i}\cap\mathcal{E}_{i+1}$, and some positive real number $\hat{r}_{i}$ which depends on $i$.
\end{theo}
\begin{proof}
Let $0\le i\le \nu-1$ and fix $\epsilon\in\mathcal{E}_{i}$. From Theorem~\ref{teo1}, the Cauchy problem (\ref{e385}), with initial conditions given by 
$$(\partial_{z}^{j}Y_{j})(T,0,\epsilon)=Y_{i,j}(T,\epsilon),\quad 0\le j\le S-1,$$
for the functions $Y_{i,j}$ defined in (\ref{e525}) admits a holomorphic solution $(T,z)\mapsto Y(T,z,\epsilon)$ defined in $S_{d_{i},\theta_{i},\Delta_{i1}|\epsilon|^{r}}\times D(0,\Delta_{i2}),$ for some $\Delta_{i1},\Delta_{i2}>0$ (recall (\ref{e411}) shows a definition of this set).

Moreover, condition b) in the construction of the initial data of the problem (\ref{e500}),(\ref{e504}), allows us to affirm this construction is also made holomorphically with respect to the perturbation parameter.

If we put $X_{i}(t,z,\epsilon)=Y(\epsilon^rt,z,\epsilon)$, then $X_{i}$ turns out to be a holomorphic function defined in $(\mathcal{T}\cap D(0,h'))\times D(0,R_0)\times \mathcal{E}_{i}$, for some $R_0,h'>0$, which turns out to be a solution of (\ref{e500}),(\ref{e504}) from its construction.

We now give proof for the estimates in (\ref{e536}). 

For every $(t,z,\epsilon)\in (\mathcal{T}\cup D(0,h'))\times D(0,R_0)\times ( \mathcal{E}_{i}\cap \mathcal{E}_{i+1})$, the difference of two solutions related to two consecutive sectors of the good covering in the perturbation parameter can be written in the form
\begin{equation}\label{e563}
X_{i+1}(t,z,\epsilon)-X_{i}(t,z,\epsilon)=\sum_{\beta\ge0}(X_{i+1,\beta}(t,\epsilon)-X_{i,\beta}(t,\epsilon))\frac{z^{\beta}}{\beta!},
\end{equation}
where 
$$X_{i,\beta}(t,\epsilon):=k\int_{L_{\gamma_{i}}}W_{\beta,i}(u,\epsilon)e^{-\left(\frac{u}{t\epsilon^r}\right)^{k}}\frac{du}{u},$$
with $(W_{i,\beta}(\tau,\epsilon))_{\beta\ge0}$ given by the recurrence (\ref{e257}), and with initial terms given by $W_{i,j}$ determined in the construction of the present Cauchy problem.

Before entering into details, it is worth mentioning the nature of the different values of $\hat{r}_{i}$, depending on $0\le i\le\nu-1$. Indeed,
\begin{equation}\label{e561}
\hat{r_{i}}\in\left\{\frac{r_2}{s_2},\frac{r_1}{s_1}\right\}.
\end{equation}
There are three different geometric situations one can find for each $0\le i\le \nu-1$:
\begin{enumerate}
\item If there are no singular directions $\frac{\pi(2j+1)+\arg(a_2)}{ks_2}$ for $j=0,...,ks_2-1$ (we will refer to such directions as singular directions of first kind) nor $\tilde{d}$ with $|\tilde{d}_{i}-\arg(d_{\mathcal{E}_{i},j})|\le\delta_{2i}$ for $j=0,...,ks_{1}$ (we will say these are singular directions of second kind) in between $\gamma_{i}$ and $\gamma_{i+1}$, then one can deform the path $L_{\gamma_{i+1}}-L_{\gamma_{i}}$ to a point by means of Cauchy theorem so that the difference $X_{i+1}-X_{i}$ is null. In this case, one can reformulate the problem by considering a new good covering combining $\mathcal{E}_{i}$ and $\mathcal{E}_{i+1}$ in a unique sector.

\item If there exists at least a singular direction of first kind but no singular directions of second kind in between $\gamma_{i}$ and $\gamma_{i+1}$, then the movable singularities depending on $\epsilon$ do not affect the geometry of the problem, whereas the path can only be deformed taking into account those singularities which do not depend on $\epsilon$. In this case $\hat{r}_i:=r_2/s_2$.

\item If there is at least a singular direction of second kind in between $\gamma_{i}$ and $\gamma_{i+1}$, then the movable singularities depend on $\epsilon$, and tend to zero. As a consequence, this affects the geometry of the problem, and the path deformation has to be made accordingly. In this case, $\hat{r}_i:=r_1/s_1$.
\end{enumerate}

Observe that Assumption (B.1) leads to $r_1/s_1<r_2/s_2$ so that the Gevrey order in the second scenary is always greater than in the third one, i.e. $\hat{r}_1\ge \hat{r}_2$.

We first consider the situation in which only singular directions of first kind appear. From c) in the construction of the initial conditions of the Cauchy problem, one can deform the integration path for the integrals in (\ref{e563}). For every $\epsilon\in\mathcal{E}_{i}\cap\mathcal{E}_{i+1}$ and $t\in\mathcal{T}\cap D(0,h')$ one has 
\begin{align*}
&X_{i+1,\beta}(t,\epsilon)-X_{i,\beta}(t,\epsilon)=k\int_{L_{\rho_{0}/2,\gamma_{i+1}}}W_{i+1,\beta}(u,\epsilon)e^{-\left(\frac{u}{t\epsilon^{r}}\right)^{k}}\frac{du}{u}\\
&-k\int_{L_{\rho_{0}/2,\gamma_{i}}}W_{i,\beta}(u,\epsilon)e^{-\left(\frac{u}{t\epsilon^{r}}\right)^{k}}\frac{du}{u}+k\int_{C(\rho_0/2,\gamma_i,\gamma_{i+1})}W_{i,i+1,\beta}(u,\epsilon)e^{-\left(\frac{u}{t\epsilon^{r}}\right)^{k}}\frac{du}{u}.
\end{align*}
Here, $L_{\rho_0/2,\gamma_{i+1}}:=[\frac{\rho_0}{2},+\infty)e^{\sqrt{-1}\gamma_{i+1}}$, $L_{\rho_0/2,\gamma_{i}}:=[\frac{\rho_0}{2},+\infty)e^{\sqrt{-1}\gamma_{i}}$ and $C(\rho_{0}/2,\gamma_i,\gamma_{i+1})$ is an arc of circle with radius $\rho_0/2$ connecting $\rho_{0}/2e^{\sqrt{-1}\gamma_{i+1}}$ and $\rho_{0}/2e^{\sqrt{-1}\gamma_{i}}$ with a well chosen orientation. Moreover, $W_{i,i+1,\beta}$ denotes the function $W_{i,\beta}$ in an open domain which contains the closed path $(L_{\gamma_{i+1}}\setminus L_{\rho_0/2,\gamma_{i+1}})-C(\rho_{0}/2,\gamma_i,\gamma_{i+1})-(L_{\gamma_{i}}\setminus L_{\rho_0/2,\gamma_{i}})$, in which $W_{i,\beta}$ and $W_{i+1,\beta}$ coincide. This is a consequence of c) in the construction of the initial data for our problem.

We first give estimates for $I_1:=k\left|\int_{L_{\rho_{0}/2,\gamma_{i}}}W_{i,\beta}(u,\epsilon)e^{-\left(\frac{u}{t\epsilon^{r}}\right)^{k}}\frac{du}{u}\right|$. The corresponding ones for $I_3:=k\left|\int_{L_{\rho_{0}/2,\gamma_{i+1}}}W_{i+1,\beta}(u,\epsilon)e^{-\left(\frac{u}{t\epsilon^{r}}\right)^{k}}\frac{du}{u}\right|$ follow the same argument, so we omit them.

$$I_{1}\le k\int_{\rho_0/2}^{\infty}|W_{i,\beta}(se^{\sqrt{-1}\gamma_{i}},\epsilon)|\exp\left(-\frac{s^{k}}{|t|^{k}|\epsilon|^{rk}}\cos(k(\gamma_{i}-\arg(t)-r\arg(\epsilon)))\right)ds.$$
Direction $\gamma_{i}$ was chosen depending on $\epsilon^{r}t$, in order that a positive real number $\Delta$ exists with $\cos(k(\gamma_{i}-t-r\arg(\epsilon)))\ge \Delta>0$, for every $\epsilon\in\mathcal{E}_{i}\cap\mathcal{E}_{i+1}$ and $t\in\mathcal{T}\cap D(0,h')$. Bearing in mind that a) in the construction of the initial conditions holds, there exist $M_0,Z_0>0$ such that
$$I_{1}\le kM_0Z_0^{\beta}\beta!\int_{\rho_0/2}^{\infty}\frac{\frac{s}{|\epsilon|^{r}}}{1+\frac{s^{2k}}{|\epsilon^{r}|^{2k}}}\exp\left(\sigma \xi(b)\frac{s^{k}}{|\epsilon|^{rk}}\right)\exp\left(-\frac{s^{k}\Delta}{|t|^{k}|\epsilon|^{rk}}\right)ds.$$
Indeed, if $h'<\left(\frac{\Delta}{\sigma\xi(b)+\Delta_1}\right)^{1/k}$ for some $\Delta_1>0$, the previous expression is upper bounded by
$$kM_0Z_0^{\beta}\beta!\int_{\rho_0/2}^{\infty}\frac{s}{|\epsilon|^{r}}\exp(-\Delta_1\frac{s^{k}}{|\epsilon|^{rk}})ds.$$
Taking into account that $k\ge2$ and $s\ge \rho_0/2$ one has $s^{2-k}\le (\rho_0/2)^{2-k}$. The previous expression equals
\begin{align}
&kM_0Z_0^{\beta}\beta!\int_{\rho_0/2}^{\infty}\frac{s^{2-k}(-k)s^{k-1}}{|\epsilon|^{r}}\exp(-\Delta_1\frac{s^{k}}{|\epsilon|^{rk}})ds\nonumber\\
&= kM_0Z_0^{\beta}\beta!|\epsilon|^{r(k-1)}\frac{(-1)}{\Delta_1 k}\int_{\rho_0/2}^{\infty}s^{2-k}\frac{(-k)s^{k-1}\Delta_1}{|\epsilon|^{rk}}\exp(-\Delta_1\frac{s^{k}}{|\epsilon|^{rk}})ds\nonumber\\
&\le kM_0Z_0^{\beta}\beta!|\epsilon|^{r(k-1)}\frac{(-1)}{\Delta_1 k}(\rho_0/2)^{2-k}\exp(-\Delta_1\frac{s^{k}}{|\epsilon|^{rk}})\left.\right|_{s=\rho_0/2}^{s\to\infty}\nonumber\\
&=kM_0Z_0^{\beta}\beta!|\epsilon|^{r(k-1)}\frac{1}{\Delta_1 k}(\rho_0/2)^{2-k}\exp(-\Delta_1(\rho_0/2)^{k}\frac{1}{|\epsilon|^{rk}})\nonumber\\
&\le M_{1}Z_0^{\beta}\beta!\exp\left(-\frac{K_1}{|\epsilon|^{rk}}\right),\label{e604}
\end{align}
for some $M_1,K_1>0$.

Analogous steps as before for the estimation of
$$I_2=k\left|\int_{C(\rho_0/2,\gamma_{i},\gamma_{i+1})}W_{i,i+1,\beta}(u,\epsilon)e^{-\left(\frac{u}{t\epsilon^{r}}\right)^{k}}\frac{du}{u}\right|$$
yield
\begin{equation}\label{e611}
I_2\le M_{2}Z_{0}^{\beta}\beta!\exp\left(-\frac{K_2}{|\epsilon|^{rk}}\right),
\end{equation}
whenever $t\in\mathcal{T}\cap D(0,h')$ for some $M_2,K_2>0$.
>From (\ref{e604}) and (\ref{e611}) one concludes there exist $M,K>0$ such that
\begin{equation}\label{e616}
|X_{i+1,\beta}(t,\epsilon)-X_{i,\beta}(,\epsilon)|\le MZ_0^\beta\beta!\exp\left(-\frac{K}{|\epsilon|^{\hat{r}_{i}}}\right),
\end{equation}
for every $\beta\ge0$, $t\in\mathcal{T}\cap D(0,h')$, $\epsilon\in\mathcal{E}_{i}\cap\mathcal{E}_{i+1}$, and with $\hat{r}_{i}=r_2/s_2$.

We now study the third situation which can occur, it is to say, that in which at least a singular direction of second kind lies in between the directions $\gamma_i$ and $\gamma_{i+1}$. Now, the coefficients appearing in the series in (\ref{e563}) are such that the integration path under consideration in the definition of the Laplace transforms is deformed in a different way. Indeed, one can write for every $\epsilon\in\mathcal{E}_{i}\cap\mathcal{E}_{i+1}$, $t\in\mathcal{T}\cap D(0,h')$, that
\begin{align*}
& X_{i+1,\beta}(t,\epsilon)-X_{i,\beta}(t,\epsilon)=k\int_{L_{\rho(|\epsilon|)/2,\gamma_{i+1}}}W_{i+1,\beta}(u,\epsilon)e^{-\left(\frac{u}{t\epsilon^{r}}\right)^{k}}\frac{du}{u}\\
&-k\int_{L_{\rho(|\epsilon|)/2,\gamma_{i}}}W_{i,\beta}(u,\epsilon)e^{-\left(\frac{u}{t\epsilon^{r}}\right)^{k}}\frac{du}{u}+k\int_{C(\rho(|\epsilon|)/2,\gamma_i,\gamma_{i+1})}W_{i,i+1,\beta}(u,\epsilon)e^{-\left(\frac{u}{t\epsilon^{r}}\right)^{k}}\frac{du}{u}.
\end{align*}
Here, the paths are  $L_{\rho(|\epsilon|)/2,\gamma_{i+1}}:=[\rho(|\epsilon|)/2,+\infty)e^{\sqrt{-1}\gamma_{i+1}}$, $L_{\rho(|\epsilon|)/2,\gamma_{i}}:=[\frac{\rho(|\epsilon|)}{2},+\infty)e^{\sqrt{-1}\gamma_{i}}$ and $C(\rho(|\epsilon|)/2,\gamma_i,\gamma_{i+1})$ is an arc of circle with radius $\rho(|\epsilon|)/2$ connecting $\rho(|\epsilon|)/2e^{\sqrt{-1}\gamma_{i+1}}$ and $\rho(|\epsilon|)/2e^{\sqrt{-1}\gamma_{i}}$ with a well chosen orientation.

We omit most of the calculs to estimate $I_4:=k\left|\int_{L_{\rho(|\epsilon|)/2,\gamma_{i+1}}}W_{i+1,\beta}(u,\epsilon)e^{-\left(\frac{u}{t\epsilon^{r}}\right)^{k}}\frac{du}{u}\right|$, $I_5:=k\left|\int_{L_{\rho(|\epsilon|)/2,\gamma_{i}}}W_{i,\beta}(u,\epsilon)e^{-\left(\frac{u}{t\epsilon^{r}}\right)^{k}}\frac{du}{u}\right|$ and  $I_6:=k\left|\int_{C(\rho(|\epsilon|)/2,\gamma_i,\gamma_{i+1})}W_{i,\beta}(u,\epsilon)e^{-\left(\frac{u}{t\epsilon^{r}}\right)^{k}}\frac{du}{u}\right|$
for they follow analogous steps as in the first case under study. Indeed, bounds for $I_4$ and $I_5$ can be obtained under the same arguments. For the study of $I_4$, one can follow the first same steps as in the estimates for $I_1$ to get that
$$I_4\le kM_2Z_0^{\beta}\beta!\exp\left(-\Delta_2\frac{\rho(|\epsilon|)^k}{|\epsilon|^{rk}}\right),$$
for some $M_2,\Delta_2>0$ not depending on $\epsilon$.
One has
$$\frac{\rho(|\epsilon|)^{k}}{|\epsilon|^{rk}}=\frac{|a_{1}|^{\frac{1}{s_1}}|\epsilon|^{\frac{s_1r_2-s_2r_1}{s_1s_2}}}{2k|\epsilon|^{\frac{r_2}{s_2}}}
=\frac{|a_{1}|^{\frac{1}{s_1}}}{2k}|\epsilon|^{-\frac{r_1}{s_1}},$$
which yields the existence of positive constants $M_3,K_3$ such that
$$I_4\le M_3Z_0^{\beta}\beta!\exp\left(-\frac{K_3}{|\epsilon|^{\frac{r_1}{s_1}}}\right),$$
for $t\in\mathcal{T}\cap D(0,h')$. We also omit the study of $I_{5}$ for the previous study can be reproduced. In view of these results, one can conclude that, in the case of a movable singularity between the arguments $\gamma_{i}$ and  $\gamma_{i+1}$, it is to say in the third case considered, one concludes there exist $M,K>0$ such that
\begin{equation}\label{e638}
|X_{i+1,\beta}(t,\epsilon)-X_{i,\beta}(t,\epsilon)|\le MZ_0^\beta\beta!\exp\left(-\frac{K}{|\epsilon|^{\hat{r}_{i}}}\right),
\end{equation}
for every $\beta\ge0$, for $t\in\mathcal{T}\cap D(0,h')$, $\epsilon\in\mathcal{E}_{i}\cap\mathcal{E}_{i+1}$, for $\hat{r}_{i}:=\frac{r_1}{s_1}$.

In view of (\ref{e616}) and (\ref{e638}), one can plug this information into (\ref{e563}) to conclude there exist $M,K>0$ such that
$$|X_{i+1}(t,z,\epsilon)-X_{i}(t,z,\epsilon)|\le M\sum_{\beta\ge0}Z_{0}^{\beta}|z|^{\beta}\exp\left(-\frac{K}{|\epsilon|^{\hat{r}_i}}\right)<M\sum_{\beta\ge0}(1/2)^{\beta}\exp\left(-\frac{K}{|\epsilon|^{\hat{r}_i}}\right),$$
for every $t\in\mathcal{T}\cap D(0,\rho_0/2)$, every $z\in D(0,1/(2Z_0))$ and all $\epsilon\in\mathcal{E}_{i}\cap\mathcal{E}_{i+1}$, for every $0\le i\le \nu-1$. This yields the result.
\end{proof}
 
\section{Existence of formal series solutions in the complex parameter and asymptotic expansions in two levels}

\subsection{ A Ramis-Sibuya theorem with two levels}\label{secrs}
The different behavior of the difference of two solutions with respect to the perturbation parameter in the intersection of adjacent sectors of the good covering studied in Theorem~\ref{teo2} provides two different levels in the asymptotic approximation of the analytic solution in the variable $\epsilon$. This behavior has also appeared in the previous work by the second author~\cite{ma0} when studying a family of singularly perturbed difference-differential nonlinear partial differential equations, where small delays depending on the perturbation parameter occur in the time variable.

\begin{defin}
Let $(\mathbb{E},\left\|\cdot\right\|_{\mathbb{E}})$ be a complex Banach space and $\mathcal{E}$ be an open and bounded sector with vertex at 0. We also consider a positive real number $\alpha$.

We say that a function $f:\mathcal{E}\to\mathbb{E}$, holomorphic on $\mathcal{E}$, admits a formal power series $\hat{f}(\epsilon)=\sum_{k\ge0}a_{k}\epsilon^{k}\in\mathbb{E}[[\epsilon]]$ as its $\alpha-$Gevrey asymptotic expansion if, for any closed proper subsector $\mathcal{W}\subseteq\mathcal{E}$ with vertex at the origin, there exist $C,M>0$ such that
$$\left\|f(\epsilon)-\sum_{k=0}^{N-1}a_{k}\epsilon^k\right\|_{\mathbb{E}}\le CM^{N}N!^{1/\alpha}|\epsilon|^{N},$$
for every $N\ge 1$, and all $\epsilon\in\mathcal{W}$.
\end{defin} 

In this section, we state a new version of the classical Ramis-Sibuya theorem (see~\cite{hssi}, Theorem XI-2-3) in two different Gevrey levels. We have decided to include the proof, which follows analogous steps as the one in~\cite{ma0} for a Gevrey level and the $1^{+}$ level, for the sake of clarity and a self-contained argumentation. In addition to this, the enunciate is written in terms of just two different Gevrey levels in order to fit our necessities, but there is no additional difficulty on considering any finite number of different levels.

\textbf{Theorem (RS)} Let $(\mathbb{E},\left\|\cdot\right\|_{\mathbb{E}})$ be a complex Banach space, and let $(\mathcal{E}_{i})_{0\le i\le \nu-1}$ be a good covering in $\C^{\star}$. We assume $G_i:\mathcal{E}_{i}\to\mathbb{E}$ is a holomorphic function for every $0\le i\le \nu-1$ and we put $\Delta_{i}(\epsilon)=G_{i+1}(\epsilon)-G_{i}(\epsilon)$ for every $\epsilon\in Z_{i}:=\mathcal{E}_{i}\cap\mathcal{E}_{i+1}$.

Here we have made the identification of the elements with index $\nu$ with the corresponding ones under index $0$. 

Moreover, we assume

\textbf{1)} The functions $G_{i}(\epsilon)$ are bounded as $\epsilon\in\mathcal{E}_{i}$ tends to the origin, for every $0\le i\le \nu-1$.

\textbf{2)} We consider $\hat{r}_1>0$ and $\hat{r}_{2}>0$, and two nonempty subsets of $\{0,...,\nu-1\}$, say $I_1$ and $I_2$, such that $I_{1}\cap I_2=\emptyset$ and $I_1\cup I_2=\{0,...,\nu-1\}$. For every $j=1,2,$ and every $i\in I_{j}$ there exist $K_i,M_i>0$ such that
$$\left\|\Delta_{i}(\epsilon)\right\|_{\mathbb{E}}\le K_i e^{-\frac{M_i}{|\epsilon|^{\hat{r}_{j}}}},$$
for every $\epsilon\in Z_i$.

Then, there exists a convergent power series $a(\epsilon)\in\mathbb{E}\{\epsilon\}$ defined on some neighborhood of the origin and $\hat{G}^{1}(\epsilon), \hat{G}^{2}(\epsilon)\in\mathbb{E}[[\epsilon]]$ such that $G_i$ can be written in the form
\begin{equation}\label{e678}
G_i(\epsilon)=a(\epsilon)+G^{1}_{i}(\epsilon)+G^{2}_{i}(\epsilon),
\end{equation}
where $G^{j}_{i}(\epsilon)$ is holomorphic on $\mathcal{E}_{i}$ and has $\hat{G}^{j}(\epsilon)$ as its $\hat{r}_{j}$-Gevrey asymptotic expansion on $\mathcal{E}_{i}$, for $j=1,2,$ and $i\in\{0,...,\nu-1\}$.

\begin{proof}
For every $0\le i\le \nu-1$ we define the holomorphic cocycles $\Delta_{i}^{j}(\epsilon)$ on the sectors $Z_i$ by
$$\Delta_{i}^{j}(\epsilon)=\Delta_{i}(\epsilon)\delta_{ij},\quad j=1,2.$$
Here, $\delta_{ij}$ is a Kronecker type function with value 1 if $i\in I_j$ and 0 otherwise.

A direct consequence of Lemma XI-2-6 from~\cite{hssi} provided by the classical Ramis-Sibuya theorem in Gevrey classes is that for every $0\le i\le \nu-1$ and for $j=1,2$, there exist holomorphic functions $\Psi_{i}^{j}:\mathcal{E}_{i}\to\C$ such that
$$\Delta_{i}^{j}(\epsilon)=\Psi_{i+1}^{j}(\epsilon)-\Psi_{i}^{j}(\epsilon)$$
for every $\epsilon\in Z_{i}$, where by convention $\Psi_{\nu}^{j}(\epsilon)=\Psi_{0}^{j}(\epsilon)$. Moreover, there exist formal power series $\sum_{m\ge0}\phi_{m,j}\epsilon^{m}\in\mathbb{E}[[\epsilon]]$ such that for each $0\le \ell\le \nu-1$ and any closed proper subsector $\mathcal{W}\subseteq\mathcal{E}_{l}$ with vertex at 0, there exist $\breve{K}_{\ell},\breve{M}_{\ell}>0$ with
$$\left\|\Psi_{\ell}^{j}(\epsilon)-\sum_{m=0}^{M-1}\phi_{m,j}\epsilon^{m}\right\|_{\mathbb{E}}\le \breve{K}_{\ell}(\breve{M}_{\ell})^{M}M!^{1/\hat{r}_{j}}|\epsilon|^{M},$$
for every $\epsilon\in\mathcal{W}$, and all positive integer $M$.

We consider the bounded holomorphic functions $a_{i}(\epsilon)=G_{i}(\epsilon)-\Psi_{i}^{1}(\epsilon)-\Psi_{i}^{2}(\epsilon),$ for every $0\le i\le \nu-1$, and $\epsilon\in\mathcal{E}_{i}$. For every $0\le i\le \nu-1$ we have 
$$a_{i+1}(\epsilon)-a_{i}(\epsilon)=G_{i+1}(\epsilon)-G_{i}(\epsilon)-\Delta_{i}^{1}(\epsilon)-\Delta_{i}^{2}(\epsilon)=G_{i+1}(\epsilon)-G_{i}(\epsilon)-\Delta_i(\epsilon)=0,$$
for $\epsilon\in Z_{i}$. Therefore, there exists a holomorphic function $a(\epsilon)$ defined on $\mathcal{U}\setminus\{0\}$, for some neighborhood of the origin $\mathcal{U}$ such that $a_{i}(\epsilon)=a(\epsilon)$ for every $0\le i\le \nu-1$. Since $a(\epsilon)$ is bounded on this domain, 0 turns out to be a removable singularity, and $a(\epsilon)$ defines a holomorphic function on $\mathcal{U}$.

Finally, one can write
$$G_{i}(\epsilon)=a(\epsilon)+\Psi_{i}^{1}(\epsilon)+\Psi_{i}^{2}(\epsilon),$$
for $\epsilon\in\mathcal{E}_{i}$, and every $0\le i\le \nu-1$. Moreover, $\Psi_{i}^{j}(\epsilon)$ admits $\hat{G}^{j}(\epsilon)=\sum_{m\ge0}\phi_{m,j}\epsilon^m$ as its $\hat{r}_{j}$-Gevrey asymptotic expansion on $\mathcal{E}_{i}$, for $j=1,2$.
\end{proof}

\textbf{Remark:} 
We put $\hat{r}_{2}:=r_{1}/s_1$ and $\hat{r}_{1}:=r_{2}/s_2$ and recall that $\hat{r}_2\le\hat{r}_1$. Assume that a sector $\mathcal{E}_{i}$ has opening a bit larger than
$\pi /\hat{r}_{1}$ and if $i \in I_{1}$ is such that
$I_{\delta_{1},i,\delta_{2}} = \{i-\delta_{1},\ldots,i,\ldots,i+\delta_{2} \} \subset I_{1}$ for some integers
$\delta_{1},\delta_{2} \geq 0$ and with the property that
\begin{equation}\label{e782}
 \mathcal{E}_{i} \subset S_{\pi/\hat{r}_{2}} \subset \bigcup_{h \in I_{\delta_{1},i,\delta_{2}}} \mathcal{E}_{h} 
 \end{equation}
where $S_{\pi/\hat{r}_{2}}$ is a sector centered at 0 with aperture a bit larger than $\pi/\hat{r}_{2}$.
Then, from the proof of Theorem (RS), we see that in the decomposition (43), the function $G_{i}^{2}(\epsilon)$ can be analytically
continued on the sector $S_{\pi/\hat{r}_{2}}$ and has the formal series $\hat{G}^{2}(\epsilon)$ as Gevrey asymptotic expansion of
order $\hat{r}_{2}$ on $S_{\pi/\hat{r}_{2}}$. Hence, $G_{i}^{2}(\epsilon)$ is the $\hat{r}_{2}-$sum of $\hat{G}^{2}(\epsilon)$
on $S_{\pi/\hat{r}_{2}}$ in the sense of the definition given in \cite{ba}, Section 3.2. Moreover, the function
$G_{i}^{1}(\epsilon)$ has $\hat{G}^{1}(\epsilon)$ as $\hat{r}_{1}-$Gevrey asymptotic expansion on $\mathcal{E}_{i}$, meaning
that $G_{i}^{1}$ is the $\hat{r}_{1}-$sum of $\hat{G}^{1}(\epsilon)$ on $\mathcal{E}_{i}$.

In other words, using the characterisation of multisummability given in \cite{ba}, Theorem 1 p. 57, the formal series $\hat{G}(\epsilon)$ is $(\hat{r}_{1},\hat{r}_{2})-$summable on $\mathcal{E}_{i}$ and its $(\hat{r}_{1},\hat{r}_{2})-$sum is the function $G_{i}(\epsilon)$ on $\mathcal{E}_{i}$.


The question that naturally arises is whether such situation can hold for certain practical situation. The answer is positive. 

Let us assume that $s_{1}=1$ and $s_{2}$ is much larger than $1$. We denote
$ap(\mathcal{E}_{i})$ the aperture of $\mathcal{E}_{i}$.
We assume that $\mathcal{T}$ is a small and thin sector that is bisected by the positive real axis. Then, from the third property in Definition~\ref{defi3}, we can assume that $ap(\mathcal{E}_{i})$ is slightly larger than $\pi/(r_{2}/s_{2})$ for some element in the good covering
$(\mathcal{E}_{i})_{0 \leq i \leq \nu-1}$. Taking into account Assumption (C), we take
$$
\frac{2\pi s_{2}}{s_{1}r_{2} - s_{2}r_{1}} > ap(\mathcal{E}_{i}) > \pi/(r_{2}/s_{2})
$$
Hence,
$$ r_{2} > \frac{s_{1}r_{2} - s_{2}r_{1}}{2} $$
>From Assumption (B.1)', we also need that $s_{1}r_{2} - s_{2}r_{1} > s_{2}$, which means under these settings that
\begin{equation}
r_{2}/s_{2} > r_{1}+1  \label{ast}
\end{equation}

Now, the consecutive ``movable'' roots of $P_{\epsilon,1}(\tau) =
\epsilon^{r_{1} - s_{1}rk}(k \tau^{k})^{s_1} + a_{1}$ are separated by an angle of
$2\pi/ks_{1}=2\pi/k$. The consecutive ``fixed'' roots of $P_{2}(\tau)$ are separated by an angle of $2\pi/(ks_{2})$.

If $s_{2}$ is much larger than 1, in between two consecutive roots of
$P_{\epsilon,1}(\tau)$ one can find at least more than two consecutive roots of $P_{2}(\tau)$ (the number of roots of $P_{2}(\tau)$ is far larger than the number of roots of $P_{\epsilon,1}(\tau)$)

We observe that the difference of any two neighboring solutions $X_{i}$, $X_{i+1}$ obtained as Laplace transform along directions $d_{i}$, $d_{i+1}$ lies between these
"fixed" roots is of exponential decay of order $r_{2}/s_{2}$. Hence, such consecutive integers $i$, $i+1$ belong to the subset $I_{1}$ (with the notation at the beginning of the remark) and the aperture of the sectors $\mathcal{E}_{i}$, $\mathcal{E}_{i+1}$ are larger than $\pi/(r_{2}/s_{2})$.
In addition, we observe that if $r_{2}/s_{2}$ is not too large compared to $r_{1}$ in (\ref{ast}), then
the union of the $\mathcal{E}_{i}$ over these aformentioned indices $i$ can contain a sector
$S_{\pi/r_{1}}$ of aperture $\pi/r_{1}$. 

In other words, we are in the configuration (\ref{e782}).

\subsection{Existence of formal power series solutions in the complex parameter}

The main result of this work states the existence of a formal power series in $\epsilon$ which can be splitted in two formal power series, each one linked to one of the different types of singularities appearing in the problem. In addition to this, the analytic solution is written as the sum of two functions which are represented by the forementioned formal power series under some Gevrey type asymptotics.

\begin{theo}\label{teopral}
Under Assumptions (A), (B) and (C) on the geometric configuration of our problem under study, and under Assumption (D), there exists a formal power series 
\begin{equation}\label{e723}
\hat{X}(t,z,\epsilon)=\sum_{\beta\ge0}H_{\beta}(t,z)\frac{\epsilon^{\beta}}{\beta!}\in\mathbb{E}[[\epsilon]],
\end{equation}
where $\mathbb{E}$ stands for the Banach space of holomorphic and bounded functions on the set $(\mathcal{T}\cap D(0,h''))\times D(0,R_0)$ equipped with the supremum norm, for some $h'',R_0>0$ provided by Theorem~\ref{teo2}, which formally solves the equation
\begin{equation}\label{e707}
(\epsilon^{r_2}(t^{k+1}\partial_t)^{s_2}+a_2)(\epsilon^{r_1}(t^{k+1}\partial_t)^{s_1}+a_1)\partial_{z}^{S}\hat{X}(t,z,\epsilon)
\end{equation}
$$=\sum_{(s,\kappa_0,\kappa_1)\in\mathcal{S}}b_{\kappa_0\kappa_1}(z,\epsilon)t^s(\partial_{t}^{\kappa_0}\partial_{z}^{\kappa_1}\hat{X})(t,z,\epsilon).
$$
Moreover, $\hat{X}$ can be written in the form
$$\hat{X}(t,z,\epsilon)=a(t,z,\epsilon)+\hat{X}^{1}(t,z,\epsilon)+\hat{X}^{2}(t,z,\epsilon),$$
where $a(t,z,\epsilon)\in\mathbb{E}\{\epsilon\}$ is a convergent series on some neighborhood of $\epsilon=0$ and $\hat{X}^{1}(t,z,\epsilon)$, $\hat{X}^{2}(t,z,\epsilon)$ are elements in $\mathbb{E}[[\epsilon]]$. Moreover, for every $0\le i\le \nu-1$, the $\mathbb{E}$-valued function $\epsilon\mapsto X_{i}(t,z,\epsilon)$ constructed in Theorem~\ref{teo2} is of the form
\begin{equation}\label{e715}
X_{i}(t,z,\epsilon)=a(t,z,\epsilon)+X_{i}^{1}(t,z,\epsilon)+X_{i}^{2}(t,z,\epsilon),
\end{equation}
where $\epsilon\mapsto X_{i}^{j}(t,z,\epsilon)$ is a $\mathbb{E}$-valued function which admits $\hat{X}^{j}(t,z,\epsilon)$ as its $\hat{r}_{j}$-Gevrey asymptotic expansion on $\mathcal{E}_{i}$, for $j=1,2$. 
\end{theo}

\begin{proof}
We consider the family of functions $(X_{i}(t,z,\epsilon))_{0\le i\le\nu-1}$ constructed in Theorem~\ref{teo2}. For every $0\le i\le \nu-1$, we define $G_{i}(\epsilon):=(t,z)\mapsto X_{i}(t,z,\epsilon)$, which turns out to be a holomorphic and bounded function from $\mathcal{E}_{i}$ into the Banach space $\mathbb{E}$ of holomorphic and bounded functions defined in $(\mathcal{T}\cap D(0,h''))\times D(0,R_0)$, for certain positive constants $R_0$ and $h''$ defined in Theorem~\ref{teo2}.

The estimates (\ref{e536}) yield that the cocycle $\Delta_{i}(\epsilon)=G_{i+1}(\epsilon)-G_{i}(\epsilon)$ satisfies exponentially flat bounds of certain Gevrey order $\hat{r}_{i}$, depending on $0\le i \le \nu-1$. Theorem (RS) guarantees the existence of formal power series $\hat{G}(\epsilon),\hat{G}^{1}(\epsilon),\hat{G}^{2}(\epsilon)\in\mathbb{E}[[\epsilon]]$ such that one has the decomposition
$$ G_{i}(\epsilon)=a(\epsilon)+\hat{G}_{i}^{1}(\epsilon)+\hat{G}_{i}^{2}(\epsilon)
$$
for $\epsilon\in\mathcal{E}_{i}$, where $G_{i}^{j}(\epsilon)$ is a holomorphic function on $\mathcal{E}_{i}$ and admits $\hat{G}^{j}_{i}(\epsilon)$ as its Gevrey asymptotic expansion of order $\hat{r}_{j}$ for all $j=1,2$. 

We define 
$$\hat{G}(\epsilon)=:\hat{X}(t,z,\epsilon)=\sum_{\beta\ge0}H_{k}(t,z)\frac{\epsilon^{k}}{k!}.$$
The proof is concluded if we show that $\hat{X}(t,z,\epsilon)$ satisfies (\ref{e707}). For any $0\le i\le \nu-1$ and $j=1,2$, the fact that $G_{i}^{j}(\epsilon)$ admits $\hat{G}_{i}^{j}(\epsilon)$ as its Gevrey expansion of some order $\hat{r}_{j}$ in $\mathcal{E}_{i}$ implies that
\begin{equation}\label{e731}
\lim_{\epsilon\to0,\epsilon\in\mathcal{E}_{i}}\sup_{(t,z)\in(\mathcal{T}\cap\{|t|<h''\})\times D(0,R_0)}|\partial^{\ell}_{\epsilon}X_{i}(t,z,\epsilon)-H_{\ell}(t,z)|=0,
\end{equation}
for every nonegative integer $\ell$. We derive $\ell>r_1+r_2$ times at both sides of equation (\ref{e707}) and let $\epsilon\to0$. From (\ref{e731}) we get a recursion formula for the coefficients in (\ref{e723}) given by
\begin{align*} a_1a_2\partial_z^{S}\left(\frac{H_{\ell}(t,z)}{\ell!}\right)&=\sum_{(s,\kappa_0,\kappa_1)\in\mathcal{S}}\sum_{m=1}^{\ell}\frac{\ell!}{m!(\ell-m)!}\frac{b_{\kappa_0\kappa_1m}(z)}{m!}\frac{\partial_{t}^{\kappa_0}\partial_{z}^{\kappa_1}H_{\ell-m}(t,z)}{(\ell-m)!}\\
&-a_2(t^{k+1}\partial_{t})^{s_1}\partial_{z}^{S}\left(\frac{H_{\ell-r_1}(t,z)}{(\ell-r_1)!}\right)-a_1(t^{k+1}\partial_{t})^{s_2}\partial_{z}^{S}\left(\frac{H_{\ell-r_2}(t,z)}{(\ell-r_2)!}\right)\\
&-(t^{k+1}\partial_{t})^{s_1+s_2}\partial_{z}^{S}\frac{H_{\ell-(r_1+r_2)}(t,z)}{(\ell-(r_1+r_2))!}.
\end{align*}

Following the same steps one concludes that the coefficients in $\hat{G}(\epsilon)$ and the coefficients of the analytic solution, written as a power in the perturbation parameter, coincide. This yields $\hat{X}(t,z,\epsilon)$ is a formal solution of (\ref{e500}), (\ref{e504}).

\end{proof}

\textbf{Remark:} In the case that $r_1-s_1rk<0$, the powers on $\epsilon$ turns out to be negative. The singularities appearing in the problem tend to infinity, and not to 0. The geometric problem that arises is different, but from our point of view, it can be solved in an analogous manner, providing singularities of two different nature as in the problem considered in the present work.


\begin{thebibliography}{99}
\bibitem{ba} W. Balser, From divergent power series to analytic functions. Theory and application of multisummable power series. Lecture Notes in Mathematics, 1582. Springer-Verlag, Berlin, 1994. x+108 pp.
\bibitem{ba2} W. Balser, Formal power series and linear systems of meromorphic ordinary differential equations. Universitext. Springer-Verlag, New York, 2000. xviii+299 pp.
\bibitem{cms} M. Canalis-Durand, J. Mozo-Fern\'andez, R. Sch\"{a}fke, \emph{Monomial summability and doubly singular differential equations}, J. Differential Equtions 233 (2007), no. 2, 485--511.
\bibitem{hssi} P. Hsieh, Y. Sibuya, \emph{Basic theory of ordinary differential equations}. Universitext. Springer-Verlag, New York, 1999.
\bibitem{ka} M. K. Kadalbajoo, K. C. Patidar, \emph{Singularly perturbed problems in partial differential equations:
a survey}. Appl. Math. Comput. 134 (2003), no. 2--3, 371–-429.
\bibitem{kami} S. Kamimoto, \emph{On the exact WKB analysis of singularly perturbed ordinary differential equations at an irregular singular point},
April, 2013, preprint RIMS--1779.
\bibitem{lamasa1} A. Lastra, S. Malek, J. Sanz, \emph{On Gevrey solutions of threefold singular nonlinear partial differential equations.} J. Differential Equations 255 (2013), no. 10, 3205--3232.
\bibitem{lama1} A. Lastra and S. Malek, \emph{On parametric Gevrey asymptotics for some nonlinear initial value Cauchy problems}, preprint 2014.
\bibitem{ma0} S. Malek, \emph{On Singularly perturbed small step size difference-differential nonlinear PDEs.} Journal of Difference Equations and Applications, (2014) vol. 20, Issue 1.
\bibitem{ma1} S. Malek, \emph{On Gevrey functions solutions of partial differential equations with fuchsian and irregular singularities}, J. Dyn. Control Syst. 15 (2009), no.2.
\bibitem{ma2} S. Malek, \emph{On the summability of formal solutionsfor doubly nonlinear partial differential equations}, J. Dyn. Control Syst. 18 (2012), no.1, 45--82. 
\bibitem{suta} K. Suzuki, Y. Takei, \emph{Exact WKB analysis and multisummability -- A case study --}, RIMS K\^{o}ky\^{u}roku, no 1861, 2013, pp. 146--155.
\bibitem{taya} H. Tahara, H. Yamazawa, \emph{Multisummability of formal solutions to the Cauchy problem for some linear partial differential equations}, Journal of Differential equations, Volume 255, Issue 10, 15 November 2013, pages 3592--3637.
\bibitem{ta} Y. Takei, \emph{On the multisummability of WKB solutions of certain singularly perturbed linear ordinary differential equations}, preprint RIMS 1803, 2014.
\end{thebibliography}
\end{document}